\documentclass[11pt,tbtags]{amsart}
\usepackage{verbatim}
\usepackage{geometry}
\usepackage{amssymb}
\usepackage{amsmath}
\usepackage{amsthm}
\usepackage{dsfont}
\usepackage[swedish, english]{babel}
\usepackage[latin1]{inputenc}
\usepackage{psfrag}
\usepackage{graphicx}

\usepackage{setspace}

\usepackage{subfigure}
\usepackage{color}
\usepackage{soul}

\theoremstyle{plain}
\newcommand{\E}{\mathbb E}
\newcommand{\R}{\mathbb R}
\newcommand{\C}{\mathcal C}

\newcommand{\LL}{\mathcal L}
\newcommand{\cU}{\mathcal U}

\newcommand{\cA}{\mathcal A}
\newcommand{\cJ}{\mathcal J}
\newcommand{\cI}{\mathcal I}

\newcommand{\cF}{\mathcal F}

\newcommand{\cT}{\mathcal T}
\newcommand{\ep}{\epsilon}

\def\P{{\mathbb P}}
\def\PP{ { \widetilde{\mathbb{P}} } }
\def\EE{ { \widetilde{\mathbb{E}} } }
\def\JJ{ { \widehat{\mathcal{J}} } }

\def\WO{ W^0  }
\def\WI{ W^1  }

\def\W{ \widetilde{W}  }

\def\cS{{\mathcal S}}

\newtheorem{theorem}{Theorem}[section]
\newtheorem{lemma}[theorem]{Lemma}
\newtheorem{corollary}[theorem]{Corollary}
\newtheorem{proposition}[theorem]{Proposition}
\newtheorem{definition}[theorem]{Definition}
\newtheorem{assumption}[theorem]{Assumption}
\newtheorem{remark}[theorem]{Remark}
\theoremstyle{definition}

\title{Dynkin games with incomplete and asymmetric information}
\author[Tiziano De Angelis, Erik Ekstr\"om and Kristoffer Glover]{Tiziano De Angelis, Erik Ekstr\"om and Kristoffer Glover}
\address{T.~De Angelis: School of Mathematics, University of Leeds, LS2 9JT Leeds, United Kingdom.}
\address{E.~Ekstr\"om: Department of Mathematics, Uppsala University, Box 480, 75106 Uppsala, Sweden.}
\address{K.~Glover: University of Technology Sydney, P.O. Box 123, Broadway, NSW 2007, Australia.}
\date{\today}
\thanks{{\em Acknowledgments}: T.~De Angelis was supported by the EPSRC grant EP/R021201/1 and  E.~Ekstr\"om
by the Knut and Alice Wallenberg Foundation.
Parts of this work were carried out during T. De Angelis and K. Glover's visits
to Uppsala University and E.~Ekstr\"om's visit to University of Leeds for the
research week ``Stopping Games for Ambiguity-Averse Players". We thank the GS Magnuson fund,
School of Maths at University of Leeds and the Heilbronn Institute for funding these visits.}

\begin{document}

\begin{abstract}
We study  the value and the optimal strategies for a two-player zero-sum optimal stopping game with incomplete and asymmetric information.
In our Bayesian set-up, the drift of the underlying diffusion process is unknown to one player (incomplete information feature), but known to the other one (asymmetric information feature).
We formulate the problem and reduce it to a fully Markovian setup where the uninformed player optimises over stopping times and the informed one uses randomised stopping times in order to hide their informational advantage.
Then we provide a general verification result which allows us to find the value of the game and players' optimal strategies by solving suitable quasi-variational inequalities with some non-standard constraints.
Finally, we study an example with linear payoffs, in which an explicit solution of the corresponding quasi-variational inequalities can be obtained.
\end{abstract}

\maketitle

\section{Introduction}

The primary focus in this paper is to devise methods to establish the existence of the value and of players' optimal strategies for two-player Dynkin games with incomplete and asymmetric information. The process underlying the game is a one-dimensional linear diffusion $X$. Both players observe the paths of $X$ and Player 2 (the {\em informed} player) knows exactly the drift and diffusion coefficient of the process.
Player 1 (the {\em uninformed} player) has incomplete information in the sense that she cannot observe directly the drift coefficient of $X$ but has a prior distribution for it and can improve upon her initial estimate by sequential observation of the process.

Crucially, the one-sided lack of information introduces an {\em asymmetry} in the game because, contrarily to the informed player,
the uninformed one cannot compute the true expected payoff of the game (for each given stopping rule).
In line with the literature on games with asymmetric information, it turns out that the informed player must use randomised stopping strategies in order to maximise the benefits of her informational advantage.
Randomisation allows the informed player to reveal information in a strategic manner to make the uninformed player act in a certain desirable way. Loosely speaking we could say that it allows the informed player to `hide' the true drift from the uniformed one in an optimal way.
On the contrary, the uninformed player cannot improve her performance by using randomisation (see Remark \ref{rem:randomise}) and therefore will simply rely on stopping times for the filtration generated by $X$.

The key contributions of the paper are: (i) we give an explicit Markovian formulation of the problem and show its equivalence with an {\em interim} version of the game (also called {\em agent-form} game), i.e., a three-player nonzero-sum game of singular control and optimal stopping (Section \ref{sec:pi}); the game is non-standard in the sense that the singular controls (played by Player 2) are not observable by Player 1; (ii) building on the previous item we formulate a
verification theorem that allows us to construct optimal strategies
for the original ({\em ex-ante}) game with incomplete and asymmetric information
(Section~\ref{sec:verif}); the verification result (Theorem~\ref{prop:verifPDE}) is formulated in terms of a quasi-variational inequality with a set of non-standard constraints; it appears that such constraints are a special feature of the asymmetric information setting; (iii) using the quasi-variational inequality approach we solve explicitly (up to numerical root-finding) a version of our game with linear payoffs (Section \ref{sec:linear}); the example illustrates how {\em reflected adjusted likelihood ratios}, introduced in Section~\ref{sec:phi*}, enable the strategical use of information in an optimal way.

To the best of our knowledge all three items above are new in the context of diffusive random dynamics. In particular, we would like to emphasise that the majority of papers on zero-sum games with continuous-time dynamics and asymmetric information focus on the existence of a value for the game, whereas the construction of
optimal strategies for both the informed and the uninformed player is mostly overlooked (we will elaborate more on this point in the literature review below). In this sense, we depart from the existing literature and present a feasible method for the characterisation of optimal strategies of the players. Moreover, we show that such optimal strategies form a Nash equilibrium in the agent-form game (interim version of the game).
As it turns out in our analysis, the informed player stops according to a generalised intensity
specified in such a way that the adjusted likelihood ratio process is {\em reflecting} along a certain boundary. The uninformed player instead stops at a hitting time of this reflected process to another boundary. We thus show that reflection of the adjusted likelihood ratio plays a vital role in Dynkin games with asymmetric information.

%Hence we believe that this paper  is a first step towards a more comprehensive study of Dynkin games with asymmetric information.

\subsection{Motivations and literature review}

Dynkin games were originally introduced in \cite{D} as a {\em game variant} of optimal stopping problems. Their popularity in the last two decades is largely due to their applications to finance. Indeed many financial contracts are equipped with exit strategies that allow one or several parties to abandon their obligations early but at an additional cost. These `exit options' embedded in the contracts are known in the mathematical finance literature as {\em game options}.

In 2000, Kifer \cite{K} showed that the arbitrage-free price of a game option can be found by solving a related zero-sum Dynkin game.
In the full information case, general conditions under which the game has a saddle point were derived in \cite{LM2} (in a martingale setting) and in \cite{EP} (in a Markovian set-up) in the case when both players
prefer the other player to stop first (so-called {\em war of attrition}).
Further studies (see \cite{LaSo}, \cite{TV}) derived the existence of a value and $\ep$-optimal strategies
for general zero-sum Dynkin games.

Acknowledging the importance of information in applications of such games, more recent literature has considered games with asymmetric information structures.
For example, asymmetric information about the time horizon of the game was considered in \cite{LM}, who concluded that, in the setting of that paper, the more you know, the longer you wait.
Gr\"un \cite{G} studied the effect of asymmetric information about the payoff structure of the game;
motivated by earlier studies (see \cite{C}, \cite{CR}) of differential games with asymmetric information as well as by an explicit example with no random dynamic, Gr\"un allowed the informed player to use randomised stopping strategies to manipulate the beliefs of the uninformed player, and she characterised the value of the game as the unique viscosity solution of a related variational inequality (see also \cite{Gens} for recent related work on differential games).
A more general situation was considered in \cite{GG}, in which each player has access to stopping times with respect to different filtrations.
In such a scenario each player must learn about the state of the world from the actions (or inaction) of the other player. Again, a variational characterisation of the value of the game is obtained in a similar form to \cite{G}. The article \cite{G} constructs optimal randomised stopping times for the informed player, whereas  \cite{GG} provides optimal randomised strategies for both players under a non-diffusive dynamics.
Note that the concept of `value' here (and in the existing literature on zero-sum games with asymmetric information) coincides with the so-called {\em ex-ante} value, i.e., before the informed player acquires the informational advantage. Once the informed player actually obtains the extra information, we obtain the {\em interim} version of the game. In the interim game the original concept of value coincides with the expected payoff of the uninformed player {\em in equilibrium} in an associated {\em agent-form} of the game (see Remarks \ref{rem:value} and \ref{rem:nzsg} for further details).

It is important to notice that the setting in \cite{G} is different from ours. In \cite{G} the observable dynamics are fully known to both players but there is asymmetry in knowledge about the payoff functions used in the game. In our problem, however, both players know the payoff functions used but there is asymmetry in knowledge about the drift of the observable dynamics. Hence, in contrast to our setting, in \cite{G} there is no learning from the observation of the process. It is also worth noticing that the variational problem in \cite{G} (and the one in \cite{GG}) looks very different from ours: Gr\"un obtains a single variational inequality (as opposed to our coupled variational problem in Theorem \ref{prop:verifPDE}) which involves three nested obstacle problems of the type `max-max-min'.
Existence of smooth solutions to such variational problems remains an open question. It does not seem trivial to show a clear connection between our variational problem and that in \cite{G}. However, our method allows us to solve an example with diffusive dynamics by proving that the associated quasi-variational inequality has a unique classical solution (see Section \ref{sec:linear}).

In \cite{EGL} a Dynkin game in which both players had differing beliefs about the drift of the underlying process was studied. However, in that article, information is fully symmetric and complete, with both players agreeing to disagree.
In comparison to \cite{EGL}, where the set-up involves no learning, and \cite{G} and \cite{GG}, where the players learn only from the actions of the opponent, our players are faced with a more complex, two-source, learning situation. In particular, the uninformed player learns about the drift of the underlying process by continuous observation of the process itself {\em and} from the actions of the informed player (or rather {\em the lack of actions}, since stopping is the only possible action).

Since learning is a key ingredient in our problem formulation, we naturally draw on the literature on stochastic filtering.
Early contributions in the area include treatments of statistical problems in sequential analysis, see for example \cite{Ba}, \cite{Ch} and \cite{S}.
A general treatment of stochastic filtering can be found in \cite{LS}, and some important early work on the application of such techniques to investment problems with incomplete information can be found in \cite{L1} and \cite{L2}.
More recent contributions along the financial lines include \cite{B},
\cite{DeA18}, \cite{M} and \cite{Z} (see also the references therein).
For optimal stopping in the context of incomplete information, an early reference is \cite{DMV} which treats the effect of incomplete information on American-style option valuation; see also \cite{Ga} and \cite{V}.
An optimal liquidation problem with unknown drift was studied in \cite{EV}, and with an unknown jump intensity in \cite{Lu}.
A two-player, zero-sum Dynkin game with symmetric and incomplete information was studied in \cite{DGV}, where optimal strategies for both players were derived.
Finally, a related paper from the economics literature is \cite{DG}, which considers the problem of a privately informed seller  trading in a market of less informed buyers, and where information about the asset's type (`good' or `bad') is gradually revealed to them. In this setting, the market places offers based on this information and on the observation of the offers rejected by the seller so far. A key difference with the current paper is that the buyers (i.e., the market) are non-strategic since the reaction of the market to new information is fully prescribed by a function of the underlying process.

\subsection{Outline of the paper}
We conclude with an outline of the material in the paper.
In Section~\ref{sec:form} we formulate the general Dynkin game and introduce the class of randomised stopping times used by the informed player.
The learning dynamics are derived and the game is reformulated as an equivalent game of stopping and singular control in Section~\ref{sec:pi}. In Section~\ref{sec:phi*} we explain how a given strategy of the informed player affects the beliefs of the uninformed one.
Next, a verification result based on quasi-variational inequalities is provided in Section~\ref{sec:verif}.
Finally, Section~\ref{sec:linear} investigates in detail an example with linear payoffs,
and Section~\ref{sec:numeric} illustrates numerically the value of the game, with a base-case set of parameters providing intuition for the optimal strategies used by the players.

\section{Setting}\label{sec:form}

Assume that on a probability space $(\Omega, \mathcal F, \mathbb P)$ we have two random variables $\theta$ and $\mathcal U$
together with a standard Wiener process $W$ mutually independent of each other, and such that
$\P(\theta=1)=\pi$ and $\P(\theta=0)=1-\pi$ where $\pi\in (0,1)$ and $\mathcal U$ is uniformly distributed on $[0,1]$.
We consider an optimal stopping game written on an underlying process $X$ with dynamics
\begin{align}\label{X00}
dX_t=\left((1-\theta)\mu_0 (X_t) + \theta \mu_1(X_t)\right)\,dt + \sigma (X_t)\,dW_t
\end{align}
on a (possibly unbounded) open interval $\mathcal I$.
Here $\mu_0(\cdot)$, $\mu_1(\cdot)$ and $\sigma(\cdot)>0$ are given Lipschitz continuous functions such that the state space of $X$ is $\mathcal I$ on both events $\{\theta=0\}$ and $\{\theta=1\}$.
Then \eqref{X00} admits a strong solution,
and to avoid further technicalities we assume that the boundary points of $\mathcal I$ are unattainable.

The game is specified by Player~$1$ choosing a (random) time $\tau$ and Player~2 choosing a (random) time $\gamma$, and at $\tau\wedge\gamma$, Player~1 receives the amount
\[R(\tau,\gamma):=f(X_{\tau})1_{\{\tau< \gamma\}} + g(X_{\gamma})1_{\{\tau\ge \gamma\}}\]
from Player~2. Here the {\em payoff functions} $f$ and $g$ are two given functions satisfying $g\geq f\geq 0$.
The objective of Player~1 (2) is to choose $\tau$ ($\gamma$) from a set of admissible stopping strategies to maximise (minimise) the expected value of $R(\tau,\gamma)$. The notion of admissible stopping strategies will be specified below.
To avoid further technical complications, we will assume continuity of the payoff functions.

\begin{assumption}\label{ass:fg}
The payoff functions $f$ and $g$ are continuous on $\cI$.
\end{assumption}

The players are rational and the model is common knowledge, i.e., both players know the functions $f$, $g$, $\mu_0$, $\mu_1$ and $\sigma$ involved. Both players observe the process $X$, but we assume that Player~1 does not know whether $\theta$ is zero or one (equivalently, whether the drift is $\mu_0$ or $\mu_1$) whereas Player 2 does. Initially, the only available information for Player~1 is the distribution of $\theta$ given above, while Player~2 knows the true value of $\theta$ already at the start of the game (the opposite case can be treated similarly).
This asymmetry is modeled by letting the information available to Player 1 be given by the augmentation with $\P$-null sets of the filtration
\[\mathcal F^X_t:=\sigma(X_s,0\leq s\leq t),\]
whereas the information available to Player 2 is given by the augmentation of the filtration
\[\mathcal F^{X,\theta}_t:=\sigma(\theta, X_s,0\leq s\leq t).\]

When considering games with asymmetric information, a crucial aspect is the strategic release of the additional knowledge from the more informed player to the less informed one.
This is modelled mathematically by allowing stopping strategies for the informed player (Player 2) to be randomised stopping times.
A priori the uninformed player (Player 1) may also use randomised stopping times but it turns out in our verification theorem (Theorem \ref{prop:verifPDE}), and its application in Section \ref{sec:linear}, that it is sufficient for her to consider just stopping times (see also Remark \ref{rem:randomise}).

The following notations will be used in the rest of the paper. We let $\cF^X=(\cF^X)_{t\ge 0}$ and $\cF^{X,\theta}=(\cF^{X,\theta}_t)_{t\ge 0}$ and denote
\begin{align*}
&\cT:=\{\text{$\tau$: $\tau$ is a $\P$-a.s. finite $\cF^X$-stopping time}\}\\
&\overline\cT:=\{\text{$\tau$: $\tau$ is an $\cF^X$-stopping time}\}\\
&\cA:=\{\text{$\Gamma$: $(\Gamma_t)_{t\ge 0-}$ is $\cF^X$-adapted and a.s. right-continuous,}\\
&\qquad\qquad\qquad\qquad\qquad \text{non-decreasing, with $\Gamma_{0-}=0$ and $\Gamma_\infty\le 1$}\}\\
&\cA^\theta:=\{\text{$\Gamma$: $(\Gamma_t)_{t\ge 0-}$ is $\cF^{X,\theta}$-adapted and a.s. right-continuous, }\\
&\qquad\qquad\qquad\qquad\qquad \text{ non-decreasing, with $\Gamma_{0-}=0$ and $\Gamma_\infty\le 1$}\}.
\end{align*}
In the definitions above we use $\Gamma_{0-}=0$ to indicate that $\Gamma_0>0$ can only be achieved by a jump of the process at time zero.

Clearly, $\cA\subseteq\cA^\theta$ and note that $\Gamma\in\cA^\theta$ if and only if
$\Gamma=\Gamma^0\mathds{1}_{\{\theta=0\}} +\Gamma^1\mathds{1}_{\{\theta=1\}}$ for some $\Gamma^0,\Gamma^1\in\cA$.
To define randomised stopping times (see, e.g., \cite{TV}), recall that $\cU$ is a random variable which is independent of $W$ and $\theta$ and Uniform(0,1)-distributed.

\begin{definition}[\textbf{Randomised stopping times}]\label{def-randomised}\

\begin{itemize}
\item A $\mathcal F^{X}$-randomised stopping time is a random variable $\gamma$ given by
\begin{equation}
\label{random}
\gamma=\inf\{t\geq 0:\Gamma_t>\cU\},\quad\text{for some $\Gamma\in\cA$.}
\end{equation}
We denote the set of $\mathcal F^{X}$-randomised stopping times by $\cT_R$.\\[+3pt]
%\vspace{+5pt}

\item A $\mathcal F^{X,\theta}$-randomised stopping time is a random variable $\gamma_\theta$ given by
\begin{equation}
\label{random2}
\gamma_\theta=\inf\{t\geq 0:\Gamma_t>\cU\},\quad\text{for some $\Gamma\in\cA^\theta$.}
\end{equation}
We denote the set of $\mathcal F^{X,\theta}$-randomised stopping times by $\cT^\theta_R$.
\end{itemize}
\end{definition}

We then have
\[\cT\subseteq\overline\cT\subseteq \cT_R\subseteq\cT^\theta_R.\]
Indeed, the first inclusion is clear by definition and the third inclusion is immediate from
$\cA\subseteq\cA^\theta$; moreover, if $\tau\in\overline\cT$, then the construction \eqref{random} with
\[\Gamma_t=\left\{\begin{array}{ll}
0 & t<\tau\\
1 & t\geq \tau\end{array}\right.\]
gives a randomised stopping time that coincides with $\tau$, which proves the middle inclusion.

Furthermore, any $\gamma_\theta\in\cT^\theta_R$ can be decomposed as
\[\gamma_\theta=\gamma_0 \mathds{1}_{\{\theta =0\}} + \gamma_1 \mathds{1}_{\{\theta =1\}}\]
for some $(\gamma_0,\gamma_1)\in\cT_R\times \cT_R$.
We say that $\gamma\in\cT_R$ is {\em generated} by $\Gamma\in\cA$ if $\gamma$ is defined as in \eqref{random}.
Similarly, $\gamma_\theta\in\cT_R^\theta$ is generated by $\Gamma\in\cA^\theta$ if $\gamma_\theta$ is defined as in \eqref{random2}.
For future reference, given a $\gamma\in\cT_R$ generated by $\Gamma\in \cA$, we also introduce $\cF^{X}$-stopping times (i.e.,
members of $\overline\cT$)
\begin{align}\label{gamma-u}
\gamma(z):=\inf\{t\geq 0:\Gamma_t>z\},\qquad \text{for all $z\in[0,1]$.}
\end{align}

\begin{definition}
A randomised stopping pair is a pair $(\tau,\gamma_\theta)\in\cT\times\cT^\theta_R$.
%A couple $(\tau,\Gamma)\in\cT\times\cA^\theta$ or a triple $(\tau, \Gamma^0,\Gamma^1)\in\cT\!\times\cA\times\cA$ are equivalent characterisations of a randomised stopping pair.
\end{definition}

With a slight abuse of notation, we sometimes write $\gamma_\theta=\Gamma=(\Gamma^0,\Gamma^1)$, where
$(\Gamma^0,\Gamma^1)$ is the
decomposition of $\Gamma$ that generates $\gamma_\theta$, and we refer also to $(\tau,\Gamma)\in\cT\times\cA^\theta$
as a randomised stopping pair.

Given a randomised stopping pair $(\tau,\gamma_\theta)\in\cT\times\cT^\theta_R$, the expected payoff of the game from the point of view of the uninformed player is
\begin{align}\label{def:J}
\cJ(\tau,\gamma_\theta)=\cJ(\tau,\Gamma^0,\Gamma^1):=\E\left[R(\tau,\gamma_\theta)\right].
\end{align}
We also say that this is the expected payoff of the {\em ex-ante} game
(see Remark \ref{rem:value} for further details around this interpretation of $\cJ$).
The lower value $\underline{v}$ and the upper value $\overline{v}$ of the game (for Player~1) are defined by
\begin{align}\label{def:v}
\underline{v}:=\sup_{\tau\in\cT}\inf_{\gamma_\theta\in\cT^\theta_R}\cJ(\tau,\gamma_\theta)\le\inf_{\gamma_\theta\in\cT^\theta_R}\sup_{\tau\in\cT}\cJ(\tau,\gamma_\theta)=:\overline{v},
\end{align}
and we say that a value $v$ exists if $\underline{v}=\overline{v}$.

\begin{definition}\label{def}
A randomised stopping pair $(\tau^*,\gamma_\theta^*)\in\cT\times\cT^\theta_R$ is a saddle point if
\begin{equation*}
\label{NE0}
\E\left[R(\tau,\gamma_\theta^*)\right] \leq \E\left[R(\tau^*,\gamma_\theta^*)\right]\leq \E\left[R(\tau^*,\gamma_\theta)\right]
\end{equation*}
for all other pairs $(\tau,\gamma_\theta)\in\cT\times\cT_R^\theta$.
\end{definition}

\begin{remark}\label{rem:interim}
For zero-sum games, it is standard to look at the notions of {\em value} of the game and players' {\em optimal strategies} (i.e.,  strategies that give a saddle point). The existence of a $(\tau^*,\gamma_\theta^*)$ for the ex-ante game
implies the existence of a value and that $\tau^*$ and $\gamma_\theta^*$ are optimal strategies for
Players 1 and 2, respectively.
Our approach will also involve the study of the {\em interim} (nonzero-sum) version of the game (or {\em agent-form} game),
for which the natural solution concept is that of Nash equilibrium.
\end{remark}

\begin{remark}
We restrict our attention to stopping times in $\cT$, i.e. stopping times that are finite $\mathbb P$-a.s.
This has the advantage that the notation and calculations become easier. Moreover, recalling that $g\ge f\ge 0$, a saddle point $(\tau^*,\gamma_\theta^*)\in \cT\times\cT^\theta_R$ (as in Definition~\ref{def}) would also be a saddle point for the corresponding game with strategies in $\overline \cT\times\cT^\theta_R$ and with expected payoff
\begin{align}\label{eq:J'}
\cJ'(\tau,\gamma_\theta):=\mathbb E[R(\tau,\gamma_\theta)\mathds 1_{\{\tau\wedge\gamma_\theta <\infty\}}].
\end{align}
Indeed, assume that $(\tau^*,\gamma^*_\theta)\in\cT\times\cT^\theta_R$ is a saddle point as in Definition \ref{def}.
By finiteness of $\tau^*$ and optimality of $\gamma^*_\theta$ we have
\[\cJ'(\tau^*,\gamma_\theta)=\cJ(\tau^*,\gamma_\theta)\ge \cJ(\tau^*,\gamma^*_\theta)= \cJ'(\tau^*,\gamma^*_\theta)\]
for all $\gamma_\theta\in\cT^\theta_R$. Moreover,
\begin{eqnarray*}
\cJ'(\tau,\gamma_\theta^*) &=&
\E\left[\liminf_{n\to\infty}R(\tau\wedge n,\gamma_\theta^*)\mathds{1}_{\{\tau\wedge\gamma_\theta^*<\infty\}}\right]
\leq \liminf_{n\to\infty}\E\left[R(\tau\wedge n,\gamma_\theta^*)\mathds{1}_{\{\tau\wedge\gamma_\theta^*<\infty\}}\right]\\
&\leq& \liminf_{n\to\infty}\E\left[R(\tau\wedge n,\gamma_\theta^*)\right]\leq \cJ(\tau^*,\gamma_\theta^*)
= \cJ'(\tau^*,\gamma_\theta^*)
\end{eqnarray*}
for any $\tau\in\overline\cT$,
where the second inequaility is by Fatou's lemma. Hence our claim is proved.
\end{remark}

\begin{remark}\label{rem:randomise}
If the game has a value ($\underline v=\overline v$ in \eqref{def:v}), there is no benefit for Player 1 in choosing a randomised stopping time (compare, e.g., \cite{LaSo}). Indeed, first note that
\[\sup_{\tau\in\overline\cT}\cJ'(\tau,\gamma_\theta)=\sup_{\tau\in\cT}\cJ(\tau,\gamma_\theta)\]
for any $\gamma_\theta\in\cT^\theta_R$ by Fatou's lemma and with $\cJ'$ as in \eqref{eq:J'} (see also the remark just above). Consequently,
for any $\gamma_\theta\in\cT^\theta_R$ and $\gamma\in\cT_R$ (that use two independent copies of $\mathcal{U}$ for the randomisation), we have
\begin{align*}
\cJ'(\gamma,\gamma_\theta)=&\int_0^1 \cJ'(\gamma(z),\gamma_\theta) \,dz \le \sup_{z\in [0,1]}\cJ'(\gamma(z),\gamma_\theta)\\
&\le \sup_{\tau\in\overline\cT}\cJ'(\tau,\gamma_\theta)=\sup_{\tau\in\cT}\cJ(\tau,\gamma_\theta),
\end{align*}
where we also recalled \eqref{gamma-u}.
The inequality above implies
\[
\underline{v}\le \sup_{\gamma\in\cT_R}\inf_{\gamma_\theta\in\cT^\theta_R}\cJ'(\gamma,\gamma_\theta)\le
\inf_{\gamma_\theta\in\cT^\theta_R}\sup_{\gamma\in\cT_R}\cJ'(\gamma,\gamma_\theta)\le \overline{v},
\]
which validates our claim, provided that $\underline{v}=\overline{v}$.
\end{remark}
Notice that the argument in the remark above requires that a value exists even if the uninformed player does not use a randomised stopping time. This situation does not occur in general but we may expect that it should hold in our game because of the combination of (at least) two facts: (1) the uninformed player (Player 1) has no information to hide and (2) she also has no incentive to avoid simultaneous stopping; indeed Player 1's payoff from simultaneous stopping is never worse than stopping on her own (due to $g\geq f$). In other words, randomisation is not required to counteract any copycat behaviour of the informed player that would force a lower payoff for the uninformed player (see \cite{LaSo}).

\begin{remark}
For bounded payoff functions $f$ and $g$, the set-up and results of the present article straightforwardly extend to the opposite case
when instead Player~1 knows the drift and Player~2 only has partial information. However, additional care is needed for unbounded payoffs; in particular, one needs to be careful with the specification of the payoff at time infinity, as well as in specifying appropriate transversality conditions as in Theorem~\ref{prop:verifPDE} below.
\end{remark}

\section{An equivalent game of stopping and singular control}\label{sec:pi}

Here we formulate the game in a Markovian setting and show that it is equivalent to a 3-player nonzero-sum game of
singular control and stopping. The latter corresponding to the {\em interim} version of the game.

We begin by rewriting the expected cost functional in a more explicit form, which takes into account Player 1's learning of the true drift through observations of the process $X$.
For $t\geq 0$ denote by
\begin{equation}
\label{Pi}
\Pi_t:=\P(\theta=1\vert \mathcal F_t^{X})
\end{equation}
the conditional probability of $\theta=1$ given observations of the underlying process $X$.
By standard filtering theory (see \cite[Chapter~9]{LS}) we have
\[dX_t=(\mu_0 (X_t)(1-\Pi_t)+ \mu_1(X_t)\Pi_t)\,dt + \sigma (X_t)\,dB_t,\qquad X_0=x\]
and
\begin{equation}
\label{Pi-dynamics}
d\Pi_t= \omega(X_t)\Pi_t(1-\Pi_t)\,dB_t,\qquad\Pi_0=\pi.
\end{equation}
Here the {\em innovation process}
\[B_t:=\int_0^t\frac{1}{\sigma(X_s)}dX_s -\int_0^t\frac{\mu_0 (X_s)+ (\mu_1(X_s)-\mu_0(X_s))\Pi_s}{\sigma(X_s)}\,ds\]
is a $(\mathbb P, \mathcal F^X)$-Brownian motion and $\omega(\cdot):=(\mu_1(\cdot)-\mu_0(\cdot))/\sigma(\cdot)$ is referred to as the {\em signal-to-noise ratio}.
Now the process $(X_t,\Pi_t)_{t\ge 0}$ is Markovian and adapted to $\cF^X$. In what follows, for $(x,\pi)\in\cI\times(0,1)$, we will denote
\[
\P_{x,\pi}(\,\cdot\,):=\P(\,\cdot\,|X_0=x,\Pi_0=\pi)\quad\text{and}\quad \E_{x,\pi}[\,\cdot\,]:=\E[\,\cdot\,|X_0=x,\Pi_0=\pi].
\]
Also, in \eqref{def:J} we use $\cJ_{x,\pi}(\tau,\gamma_\theta)$ to emphasise the dependence of the expected game payoff on the initial data.

In preparation for the reduction of our game to one of control and stopping, we introduce integrals of the form
\[
\int_0^\tau Y_{t-}\, d\Gamma_t:=Y_0\Gamma_0+\int_{(0,\tau]} Y_{t-}\, d\Gamma_t,
\]
for $\Gamma\in\cA$ and $Y$ a right-continuous, non-negative process adapted to $\cF^X$.
Integrals of this type are to be interpreted in the Lebesgue-Stieltjes sense, and it is important to remark that, in this context,
%we denote
%\[
%\int_0^\tau\ldots d\Gamma_t:= \int_{[0,\tau]}\ldots d\Gamma_t
%\]
%so that
both the (possible) initial and terminal jumps of the process $\Gamma$ are accounted for.
Moreover, recalling \eqref{gamma-u} and using \cite[Prop.~4.9, Ch.~0]{RY}, we have
\begin{align}\label{RY-int}
\int_0^1 g(X_{\gamma(z)})\mathds{1}_{\{\gamma(z)\le \tau\}}dz =\int_0^\tau\! g(X_t)d\Gamma_t
\end{align}
for $\tau\in\cT$.

\begin{proposition}
\label{prop:J}
For $(x,\pi)\in\cI\times(0,1)$ and any $(\tau,\gamma_\theta)\in\cT\times\cT^\theta_R$ we have
\begin{align}\label{J}
\cJ_{x,\pi} (\tau,\gamma_\theta)
=&\,\E_{x,\pi}\left[(1-\Pi_\tau)(1-\Gamma^0_\tau)f(X_\tau) +(1- \Pi_\tau)\int_0^\tau g(X_t)d\Gamma^0_t\right]\\
& + \E_{x,\pi}\left[\Pi_\tau (1-\Gamma^1_\tau)f(X_\tau) + \Pi_\tau\int_0^\tau g(X_t)d\Gamma^1_t\right],\nonumber
\end{align}
where $(\Gamma^0,\Gamma^1)\in\cA\times\cA$ is the couple that generates $\gamma_\theta$.
\end{proposition}

\begin{proof}
By definition of the game's payoff and by the definition of $\cT^\theta_R$ we have
\begin{eqnarray}\label{J00}
\cJ_{x,\pi}\,(\tau,\gamma_\theta) &=&\E_{x,\pi}\left[f(X_\tau)\mathds{1}_{\{\tau<\gamma_\theta\}}+ g(X_{\gamma_\theta})\mathds{1}_{\{\gamma_\theta\le \tau\}}\right]\nonumber\\[+4pt]
&=&\E_{x,\pi}\left[f(X_\tau)\mathds{1}_{\{\tau<\gamma_0\}\cap\{\theta=0\}}+g(X_{\gamma_0})\mathds{1}_{\{\gamma_0 \le \tau\}\cap\{\theta=0\}}\right]\\[+4pt]
&&+\E_{x,\pi}\left[f(X_\tau)\mathds{1}_{\{\tau<\gamma_1\}\cap\{\theta=1\}}+g(X_{\gamma_1})\mathds{1}_{\{\gamma_1 \le \tau\}\cap\{\theta=1\}}\right].\nonumber
\end{eqnarray}
With the aim of using the tower property in the expression above, we claim that
\begin{align}
\label{p1}&\E_{x,\pi}\left[f(X_\tau)\mathds{1}_{\{\tau<\gamma_0\}\cap\{\theta=0\}}\big|\cF^X_\tau\right]
= (1-\Pi_\tau)f(X_\tau)(1-\Gamma^0_\tau) \\[+4pt]
\label{p2}&\E_{x,\pi}\left[f(X_\tau)\mathds{1}_{\{\tau<\gamma_1\}\cap\{\theta=1\}}\big|\cF^X_\tau\right]
=\Pi_\tau f(X_\tau)(1-\Gamma^1_\tau)\\[+4pt]
\label{p3}&\E_{x,\pi}\left[g(X_{\gamma_0})\mathds{1}_{\{\gamma_0 \le \tau\}\cap\{\theta=0\}}\big|\cF^X_\tau\right]=(1-\Pi_\tau)\int_0^\tau g(X_t)d\Gamma^0_t\\
\label{p4}&\E_{x,\pi}\left[g(X_{\gamma_1})\mathds{1}_{\{\gamma_1 \le \tau\}\cap\{\theta=1\}}\big|\cF^X_\tau\right]
=\Pi_\tau\int_0^\tau g(X_t)d\Gamma^1_t.
\end{align}
Taking conditional expectation inside \eqref{J00} and using the above expressions we obtain \eqref{J}.

It therefore only remains to prove the formulae above. Let us start by noticing that
\begin{equation}
\label{inclusions}
\{\Gamma^0_\tau<\cU \}\subseteq \{\tau<\gamma_0\}\subseteq \{\Gamma^0_\tau\leq\cU \}.
\end{equation}
Since $X_{\tau}$ is $\cF^X_\tau$-measurable, using simple properties of conditional expectation and \eqref{Pi} we have
\begin{align*}
\E_{x,\pi}\left[f(X_\tau)\mathds{1}_{\{\tau<\gamma_0\}\cap\{\theta=0\}}\big|\cF^X_\tau\right]=f(X_\tau)\P_{x,\pi}\left(\tau<\gamma_0\big|\cF^X_\tau,\theta=0\right)(1-\Pi_\tau).
\end{align*}
Then, by definition of $\gamma_0$, using that $\cU$ is independent of $\theta$, $\Gamma^0_\tau$ is $\cF^X_\tau$-measurable and
\eqref{inclusions}, we also obtain
\[\P_{x,\pi}\left(\tau<\gamma_0\big|\cF^X_\tau,\theta=0\right)=
\P_{x,\pi}\left(\Gamma^0_\tau\le \cU \big|\cF^X_\tau,\theta=0\right) = (1-\Gamma^0_\tau).\]
Combining the last two expressions leads to \eqref{p1}. Clearly \eqref{p2} follows by the same argument.

For \eqref{p3} we follow a similar approach and we also recall $\gamma(u)$ as in \eqref{gamma-u} and \eqref{RY-int}. Then we have
\begin{eqnarray}\label{p5}
\E_{x,\pi}\left[g(X_{\gamma_0})\mathds{1}_{\{\gamma_0\le \tau\}\cap\{\theta=0\}}\big|\cF^X_\tau\right]%\nonumber[+3pt]
&=&\E_{x,\pi}\left[g(X_{\gamma_0})\mathds{1}_{\{\gamma_0\le \tau\}}\big|\cF^X_\tau,\theta=0\right](1-\Pi_\tau)\\[+3pt]
&=&\E_{x,\pi}\left[\int_0^1 g(X_{\gamma_0(z)})\mathds{1}_{\{\gamma_0(z)\le \tau\}}dz\big|\cF^X_\tau,\theta=0\right](1-\Pi_\tau)\nonumber\\
&=&(1-\Pi_\tau)\int_0^1 g(X_{\gamma_0(z)})\mathds{1}_{\{\gamma_0(z)\le \tau\}}dz \nonumber\\
&=&(1-\Pi_\tau)\int_0^\tau g(X_t)d\,\Gamma^0_t\,,\notag
\end{eqnarray}
where in the penultimate equality we used that $g(X_{\gamma_0(z)})\mathds{1}_{\{\gamma_0(z)\le \tau\}}$ is $\cF^X_\tau$-measurable for all $z\geq 0$, and the last equality is due to \eqref{RY-int}.

The proof of \eqref{p4} is analogous.
\end{proof}

\begin{remark}\label{rem:intuition}
Intuitively the expression in \eqref{J} can be interpreted as follows: imagine the informed player announces the $\cF^{X,\theta}$-randomised stopping strategy $\Gamma$ that she intends to use; then the uninformed player (or any other external observer) can evaluate the expected payoff associated to any choice of stopping time $\tau\in\cT$ and any sample path $t\mapsto X_t(\omega)$ of the underlying process. In particular, given $\tau\in\cT$ the term $(1-\Pi_\tau)$ is the probability associated to $\{\theta=0\}$ based on the observation of the path of $X$ until time $\tau$, while $(1-\Gamma^0_\tau)$ is the probability that the informed player does not stop before time $\tau$ if the event $\{\theta=0\}$ has occurred. Therefore, for a given $\omega\in\Omega$, the quantity $(1-\Pi_\tau(\omega))(1-\Gamma^0_\tau(\omega))$ represents the probability that the uninformed player will stop before the informed one on the event $\{\theta=0\}$. A symmetric argument can be applied to the term $\Pi_\tau(1-\Gamma^1_\tau)$, which is the probability that Player 1 stops before Player 2 on the event $\{\theta=1\}$. Combining the two, for each $\omega\in\Omega$ Player 1 has a probability $(1-\Pi_\tau(\omega))(1-\Gamma^0_\tau(\omega))+\Pi_\tau(\omega)(1-\Gamma^1_\tau(\omega))$ to stop before Player 2 and to receive $f(X_\tau(\omega))$.

Analogous considerations can be applied to the integral terms. On the event $\{\theta=i\}$, $i=0,1$, the increment $d\Gamma^i_t$ measures the probability that Player 2 stops during the (infinitesimal) time interval $[t,t+dt)$. Then, for each $\omega\in\Omega$, the integral
\[
\int_0^\tau g(X_t(\omega))d\Gamma^i_t(\omega),\quad i=0,1,
\]
is the accumulated expected
payoff received by Player 1 prior to time $\tau(\omega)$ if $\theta(\omega)=i$.
\end{remark}

It will be convenient in what follows to also use the {\em likelihood ratio} process $\Phi_t:=\Pi_t/(1-\Pi_t)$,
whose dynamics under $\P$ are derived from \eqref{Pi-dynamics} and It\^o's formula as
\begin{align}\label{Phi00}
\frac{d\Phi_t}{\Phi_t}=\omega(X_t)\left(dB_t+\Pi_t\, \omega(X_t)dt\right),\qquad \Phi_0=\varphi,
\end{align}
where $\varphi=\pi/(1-\pi)$.
The dynamics of the two-dimensional diffusion $(X,\Phi)$ are somewhat involved under $\P$, and we prefer instead to use the measures $\P^0$ and $\P^1$ specified by
\[\P^i(A):= \P(A\,\vert\,\theta =i)\]
for $A\in\mathcal F^X_\infty$.
It is well-known (see \cite[Chapter 9]{LS}) that
\begin{align}
\label{P0}&\frac{d \P^0}{d\P}\Big|_{\mathcal F^X_t}= \frac{1-\Pi_t}{1-\pi} = \frac{1+\varphi}{1+\Phi_t}
=\exp\left(-\tfrac{1}{2}\int_0^t\omega^2(X_s)\Pi^2_s ds-\!\int_0^t\omega(X_s)\Pi_s \,dB_s\right),\\[+3pt]%\notag
\label{P1}&\frac{d\P^1}{d\P}\Big|_{\cF^X_t}=\frac{\Pi_t}{\pi}=\exp\left(-\tfrac{1}{2}\int_0^t\omega^2(X_s)(1-\Pi_s)^2 ds+\!\int_0^t\omega(X_s)(1-\Pi_s) dB_s\right),%\notag
\end{align}
and that $X$ and $\Phi$ satisfy
\begin{equation}\label{hatPdynamics}
\left\{\begin{array}{l}
dX_t=\mu_i (X_t)\,dt + \sigma (X_t)\,d W^i_t\\
d\Phi_t=\omega(X_t)\Phi_t\,d\WO_t\\
\hspace{7mm}=\omega^2(X_t)\Phi_tdt+\omega(X_t)\Phi_t\,d\WI_t, \end{array}\right.
\end{equation}
where
\[W^i_t:=-\int_0^t\omega(X_s)(i-\Pi_s)\,ds+B_t\]
is a $\P^i$-Brownian motion, for $i=0,1$.
Note that the system \eqref{hatPdynamics} is semi-decoupled in the sense that the dynamics of
$X$ do not depend on $\Phi$. Also notice that
\begin{align}\label{eq:P1P0}
\Phi_t=\frac{d\P^1}{d\P^0}\Big|_{\cF^X_t},\quad\text{for $t\in[0,\infty)$},
\end{align}
by \eqref{P0} and \eqref{P1}.

We now rewrite our problem under the measure $\P^0$.
In what follows we set $\E^i[\,\cdot\,]$ for the expectation under the measure $\P^i$, with $i=0,1$.

\begin{corollary} {\bf (The expected payoff for the uninformed player.)}
\label{game-reduced}
For $(x,\pi)\in\cI\times(0,1)$ and any $(\tau,\gamma_\theta)\in\cT\times\cT^\theta_R$ we have
\begin{align}\label{JJ1}
\cJ_{x,\pi}(\tau,\gamma_\theta) =&\frac{1}{1+\varphi} \left(\E_{x,\pi}^0\left[(1-\Gamma^0_\tau)f(X_\tau) + \int_0^\tau  g(X_t)d\Gamma^0_t\right]\right.\\
&\hspace{+20pt} + \left.\E_{x,\pi}^0\left[(1-\Gamma^1_\tau)\Phi_\tau f(X_\tau) +\int_0^\tau \Phi_t g(X_t)d\Gamma^1_t\right]\right),\nonumber
\end{align}
where $\varphi=\pi/(1-\pi)$.
\end{corollary}

\begin{proof}
We start by looking at the first term on the right-hand side of \eqref{J}. For any $\tau\in\cT$, we recall that $\Pi_\tau=\P(\theta=1|\cF_\tau)$, so by the tower property and the definition of $\P^0$ we get
\begin{align}\label{change0}
\E\left[(1-\Pi_\tau)(1-\Gamma^0_\tau)f(X_\tau)\right]&=\E\left[(1-\Gamma^0_\tau)f(X_\tau)1_{\{\theta=0\}}\right]\\
&= \E\left[(1-\Gamma^0_\tau)f(X_\tau)|\theta=0\right]\P(\theta=0)\notag\\
&= (1-\pi)\E^0\left[(1-\Gamma^0_\tau)f(X_\tau)\right].\notag
\end{align}
By the same argument we also obtain
\[
\E\left[(1-\Pi_\tau)\int_0^\tau g(X_t)d\Gamma^0_t\right]=(1-\pi)\E^0\left[\int_0^\tau g(X_t)d\Gamma^0_t\right].
\]

For the remaining terms in \eqref{J} we notice first that
\begin{align*}
\E\left[\Pi_\tau(1-\Gamma^1_\tau)f(X_\tau)\right]&=\E\left[(1-\Pi_\tau)(1-\Gamma^1_\tau)\Phi_\tau f(X_\tau)\right]\\
&=(1-\pi)\E^0\left[(1-\Gamma^1_\tau)\Phi_\tau f(X_\tau)\right].
\end{align*}
Second, setting $g_n:=n\wedge g$ and $\tau_m:=\inf\{t\ge 0: \Pi_t\ge m/(m+1)\}\wedge\tau\wedge m$ we have
\begin{align*}
\E\left[\Pi_\tau\int_0^\tau\!\! g(X_t)d\Gamma^1_t\right]=&\lim_{n\to\infty}\E\left[\Pi_\tau\int_0^\tau\!\! g_n(X_{t})d\Gamma^1_{t}\right]
=\lim_{n\to\infty}\lim_{m\to\infty}\E\left[\Pi_{\tau_m}\int_0^{\tau_m}\!\! g_n(X_{t})d\Gamma^1_{t}\right]
\end{align*}
where the first equality holds by monotone convergence and the second one by dominated convergence. Then, for fixed $n,m>0$, we have
\begin{align*}
\E\left[\Pi_{\tau_m}\int_0^{\tau_m}\!\! g_n(X_{t})d\Gamma^1_{t}\right]=\E\left[(1-\Pi_{\tau_m})\Phi_{\tau_m}\int_0^{\tau_m}\!\! g_n(X_{t})d\Gamma^1_{t}\right]=(1-\pi)\E^0\left[\Phi_{\tau_m}\int_0^{\tau_m}\!\! g_n(X_{t})d\Gamma^1_{t}\right],
\end{align*}
by the same argument as in \eqref{change0}. The process $(\Phi_{t\wedge\tau_m})_{t\ge 0}$ is a continuous $\P^0$-martingale with values in $(0,m]$ and moreover
\[
0\le \int_0^{\tau_m\wedge s}\!\! g_n(X_{t})d\Gamma^1_{t}\le n\quad\text{for all $s\ge 0$}.
\]
Then, by Ito's formula we have
\[
\E^0\left[\Phi_{\tau_m}\int_0^{\tau_m}\!\! g_n(X_{t})d\Gamma^1_{t}\right]=\E^0\left[\int_0^{\tau_m}\!\! \Phi_{t}\,g_n(X_{t})d\Gamma^1_{t}\right].
\]
The latter implies
\[
\E\left[\Pi_\tau\int_0^\tau\!\! g(X_t)d\Gamma^1_t\right]=(1-\pi)\lim_{n\to\infty}\lim_{m\to\infty}\E^0\left[\int_0^{\tau_m}\!\! \Phi_{t}\,g_n(X_{t})d\Gamma^1_{t}\right]=(1-\pi)\E^0\left[\int_0^{\tau}\!\! \Phi_{t}\,g(X_{t})d\Gamma^1_{t}\right]
\]
by monotone convergence.
Combining the above expressions we obtain \eqref{JJ1} upon noticing that $1-\pi=(1+\varphi)^{-1}$.
\end{proof}
The expression in \eqref{JJ1} has the same intuitive meaning as explained in Remark \ref{rem:intuition} but with the likelihood ratio in place of the probabilities $\Pi_\tau$ and $1-\Pi_\tau$.
The next corollary follows in a similar way using \eqref{P0} and \eqref{P1} in the first and second term on the right-hand side of \eqref{J}, respectively.

\begin{corollary}
\label{informed-player}{\bf (The expected cost for the informed player.)}
For $(x,\pi)\in\cI\times(0,1)$ and any $(\tau,\gamma_\theta)\in\cT\times\cT^\theta_R$ we have
\begin{eqnarray}\label{Javg}
\mathcal J_{x,\pi}(\tau,\gamma_\theta) &=& (1-\pi)\mathcal J^0_{x,\pi}(\tau,\Gamma^0) + \pi \mathcal J^1_{x,\pi}(\tau,\Gamma^1) ,
\end{eqnarray}
where
\begin{align}\label{J0}
\mathcal J^0_{x,\pi}(\tau,\Gamma^0) := \E^0_{x,\pi}\left[(1-\Gamma^0_\tau)f(X_\tau) + \int_0^\tau g(X_t)d\Gamma^0_t\right]
\end{align}
and
\begin{align}\label{J1}
\mathcal J^1_{x,\pi}(\tau,\Gamma^1) := \E^1_{x,\pi}\left[(1-\Gamma^1_\tau)f(X_\tau) + \int_0^\tau g(X_t)d\Gamma^1_t\right].
\end{align}
\end{corollary}
\begin{remark}\label{rem:value}
The expression in \eqref{Javg} offers the following interpretation of the functional $\cJ_{x,\pi}$. Imagine that before the game starts (i.e., at time $t=0-$), neither of the players knows $\theta$. However, they both know that as soon as the game starts (i.e., at time $t=0$) Player 2 will learn the true value of $\theta$. Then, we can think of $\cJ_{x,\pi}$ as the expected payoff for both players at time $t=0-$ (given the randomised stopping pair $(\tau,\gamma_\theta)$).  As one would expect in this context, the payoff at time $t=0-$ is the average according to the prior distribution of $\theta$ of the payoffs in the two possible scenarios (the {\em ex-ante} payoff of the game).

As soon as the game starts at time $t=0$, the payoff of the informed player `collapses' into either $\cJ^0_{x,\pi}$ or $\cJ^1_{x,\pi}$ because she learns the true value of $\theta$. On the contrary, the expected payoff of Player~1 remains $\cJ_{x,\pi}$. This situation corresponds to the {\em interim} version of the game, i.e., after the information has been acquired by the informed player.

It is worth noting that many papers in the literature on asymmetric games with continuous-time dynamics (see, e.g., \cite{C,CR,GG,G}) only use the payoff $\cJ_{x,\pi}$ for their analysis. The `value' of the game in those papers corresponds in our setting to the expected equilibrium payoff for the uninformed player in the {\em interim} game or, equivalently, to the value of the game {\em ex-ante}.
\end{remark}

We now come to the final formulation of the game's expected payoff, which is also the one that we find most convenient for our solution method. For $(\tau,\gamma_\theta)\in\cT\times\cT^\theta_R$ and $\varphi=\pi/(1-\pi)$, let us denote
\begin{align}\label{JJ}
\JJ_{x,\varphi}(\tau,\gamma_\theta):=(1+\varphi)\cJ_{x,\pi}(\tau,\gamma_\theta).
\end{align}
The next result connects a saddle point for
the ex-ante version of our game with a Nash equilibrium for its interim version (in the same spirit as in \cite{H}).

\begin{proposition}\label{prop:NE}
Let $(x,\varphi)\in\cI\times\R_+$ be given. A randomised stopping pair $(\tau^*,\gamma^*_\theta)\in\cT\times\cT^\theta_R$ is a saddle point in the {\em ex-ante} game (Definition \ref{def}) if and only if it is a Nash equilibrium in the {\em agent-form} game. That is, if and only if, letting $(\Gamma^{*,0},\Gamma^{*,1})\in\cA\times\cA$ be the couple that generates $\gamma^*_\theta$, we have
\begin{equation}
\label{NE2}
\mathcal J^0_{x,\varphi}(\tau^* ,\Gamma^{*,0}) \leq \mathcal J^0_{x,\varphi}(\tau^* ,\Gamma^0),
\end{equation}
\begin{equation}
\label{NE3}
\mathcal J^1_{x,\varphi}(\tau^* ,\Gamma^{*,1}) \leq \mathcal J^1_{x,\varphi}(\tau^* ,\Gamma^1)
\end{equation}
and
\begin{equation}
\label{NE1}
\JJ_{x,\varphi}(\tau,\gamma^*_\theta)\leq \JJ_{x,\varphi}(\tau^*,\gamma^*_\theta)
\end{equation}
for all randomised stopping pairs $(\tau,\Gamma)\in \cT\times\cA^\theta$.
\end{proposition}

\begin{proof}
We have from \eqref{Javg} that
\begin{align}\label{JJ2}
\JJ_{x,\varphi}(\tau,\gamma_\theta)=\cJ^0_{x,\varphi}(\tau,\Gamma^0)+\varphi\cJ^1_{x,\varphi}(\tau,\Gamma^1).
\end{align}
It follows that a strategy
$\Gamma=(\Gamma^0,\Gamma^1)$ of the informed player minimises $\JJ_{x,\varphi}(\tau^*,\gamma_\theta)$
if and only if $\Gamma^0$ and $\Gamma^1$ minimise
$\mathcal J^0_{x,\varphi}(\tau^*,\Gamma^0)$ and $\mathcal J^1_{x,\varphi}(\tau^*,\Gamma^1)$, respectively. Condition \eqref{NE1} instead is the same as in Definition \ref{def}.
\end{proof}

For a Nash equilibrium $(\tau^*,\gamma^*_\theta)$ we refer to $\JJ_{x,\varphi}(\tau^*,\gamma^*_\theta)$, $\cJ^0_{x,\varphi}(\tau^* ,\Gamma^{*,0})$ and
$\cJ^1_{x,\varphi}(\tau^* ,\Gamma^{*,1})$ as the corresponding equilibrium payoffs.

\begin{remark}\label{rem:nzsg}
We observe that Proposition \ref{prop:NE} gives an interpretation of the game as a 3-player nonzero-sum game between a stopper and two controllers.
Notice that the stopper plays simultaneously against both controllers, whereas each controller only plays against the stopper. This is in parallel with classical results on games with incomplete information (see \cite{H} or \cite{AM}). In particular, the 3-player game described above can be interpreted as the `agent-form' of our (Bayesian) game, with \eqref{NE2}--\eqref{NE1} representing the `interim definition' of equilibrium.
\end{remark}

The next proposition provides a condition under which, if a Nash equilibrium exists in the agent-form game, it is possible to find another Nash equilibrium such that the informed player (Player 2) never stops on the event $\{\theta=0\}$. This result is useful to simplify the construction of a Nash equilibrium in certain cases, for example in the problem studied in Section~6.

\begin{proposition}\label{prop:G0}
Fix $(x,\varphi)\in\R_+\times\R_+$ and assume that $(\tau^*,\Gamma^{*,0},\Gamma^{*,1})\in\cT\times\cT^\theta_R$ is a
Nash equilibrium in the agent-form game such that
\begin{equation}
\label{cond}
 \cJ^0(\tau,0)\leq \cJ^0(\tau,\Gamma^{*,0})
\end{equation}
for all $\tau\in\cT$. Then $(\tau^*,0,\Gamma^{*,1})$ is also a Nash equilibrium.
\end{proposition}

\begin{proof}
First note that \eqref{NE3} holds since $(\tau^*,\Gamma^{*,0},\Gamma^{*,1})$ is a
Nash equilibrium. Moreover, for any $\Gamma^0\in\cT_R$,
\begin{equation}
\label{res}
\mathcal J^0_{x,\varphi}(\tau^* ,0) \leq
\mathcal J^0_{x,\varphi}(\tau^* ,\Gamma^{*,0})\leq \mathcal J^0_{x,\varphi}(\tau^* ,\Gamma^0),
\end{equation}
where the first inequality comes from \eqref{cond} and the second from \eqref{NE2}.
Thus \eqref{NE2} and \eqref{NE3} hold for the candidate equilibrium $(\tau^*,0,\Gamma^{*,1})$.

It remains to show that \eqref{NE1} holds for the candidate equilibrium. To do that, note first that
inserting $\Gamma^0=0$ in \eqref{res} yields
\[\mathcal J^0_{x,\varphi}(\tau^* ,0)=\mathcal J^0_{x,\varphi}(\tau^* ,\Gamma^{*,0}).\]
Consequently,
\[\JJ_{x,\varphi}(\tau^* ,0,\Gamma^{*,1})=\JJ_{x,\varphi}(\tau^* ,\Gamma^{*,0},\Gamma^{*,1}),\]
so
\begin{align*}
\JJ_{x,\varphi}(\tau ,0,\Gamma^{*,1}) \leq \JJ_{x,\varphi}(\tau ,\Gamma^{*,0},\Gamma^{*,1})\leq
\JJ_{x,\varphi}(\tau^* ,\Gamma^{*,0},\Gamma^{*,1})=\JJ_{x,\varphi}(\tau^* ,0,\Gamma^{*,1})
\end{align*}
for $\tau\in\cT$, where the first inequality follows from \eqref{cond} and the second one from $(\tau^*,\Gamma^{*,0},\Gamma^{*,1})$
being a Nash equilibrium. This completes the proof.
\end{proof}

\section{Adjusted beliefs and Nash equilibria}\label{sec:phi*}

If an equilibrium exists in the agent-form game, then both players are able to compute it, in the sense that they both know the stopping time $\tau^*$ and the increasing processes $\Gamma^{*,0}$ and $\Gamma^{*,1}$ that are used to generate $\gamma^*_\theta$ (we do not consider the question of uniqueness of equilibria in this paper). Given the generating processes $\Gamma^{*,0}$  and $\Gamma^{*,1}$, the uninformed player calculates what we refer to as the \emph{adjusted posterior probability}
\begin{align}\label{gammath}
\Pi^{*}_t:=\P(\theta=1\big|\cF^X_t,\gamma^*_\theta>t),\qquad t\ge 0.
\end{align}
Thus, while the posterior probability $\Pi_t$ is only based on the observation of the sample path of $X$, the adjusted posterior probability also takes into account
an assumed strategy of the informed player.

Using properties of conditional expectations we can write
\begin{equation}\label{pg0}
\Pi^{*}_t
=\frac{\P(\theta=1, \gamma^*_\theta>t\big|\cF^X_t)}{\P(\gamma^*_\theta>t\big|\cF^X_t)}
=\frac{\P(\gamma^*_1>t\big|\cF^X_t,\theta=1)\P(\theta=1\big|\cF^X_t)}{\P(\gamma^*_\theta>t\big|\cF^X_t)}=\frac{(1-\Gamma^{*,1}_t)\Pi_t}{\P(\gamma^*_\theta>t\big|\cF^X_t)},
\end{equation}
where the last equality is obtained using the same arguments as those used in the proof of Proposition~\ref{prop:J}.
Similarly, for the denominator we have
\begin{align}\label{pg2}
\P(\gamma^*_\theta>t\big|\cF^X_t)
&=\P(\gamma^*_0>t\big|\cF^X_t,\theta=0)\P(\theta=0\big|\cF^X_t) +\P(\gamma^*_1>t\big|\cF^X_t,\theta=1)\P(\theta=1\big|\cF^X_t)\\
&=(1-\Gamma^{*,0}_t)(1-\Pi_t)+(1-\Gamma^{*,1}_t)\Pi_t.\nonumber
\end{align}
Combining \eqref{pg0}--\eqref{pg2} gives
\begin{align}\label{pg3}
\Pi^{*}_t=\frac{(1-\Gamma^{*,1}_t)\Phi_t}{1-\Gamma^{*,0}_t+(1-\Gamma^{*,1}_t)\Phi_t},
\end{align}
and then it becomes straightforward to see that the adjusted posterior probability satisfies
\begin{align}\label{eq:phi*}
\Phi^{*}_t:=\frac{\Pi^{*}_t}{1-\Pi^{*}_t}=\Phi_t \frac{1-\Gamma^{*,1}_t}{1-\Gamma^{*,0}_t}, \qquad t\ge 0.
\end{align}
Thus $\Phi_t^*$ is the likelihood ratio of the adjusted posterior probability, or the {\em adjusted likelihood ratio}.

There is a subtle point, whose understanding is key to the proof of Theorem \ref{prop:verifPDE} below. When constructing Nash equilibria in the agent-form game we will need, in particular, to verify conditions \eqref{NE2} and \eqref{NE3}; so the question arises as to what $\tau^*$ should depend on. First and foremost, we recall that processes $\Gamma\in\cA^\theta$ are not observable by Player~1 because the two players do not communicate (they only see their opponent stop at some point). Therefore, if Player 2 plays a non-equilibrium pair $(\Gamma^0,\Gamma^1)$ the stopping time $\tau^*$ is not affected by this (sub-optimal) choice. However, both players know the equilibrium pair $(\tau^*,\Gamma^*)$. Hence, Player 1's choice of $\tau^*$ may at most depend (and it will) on the adjusted belief process associated to $\Gamma^*$. That is, one should expect that $\tau^*$ is a stopping time for the paths of the process $(X,\Phi^*)$. Hence we are naturally led to consider equilibria in {\em open-loop} strategies.

\section{A verification result}\label{sec:verif}

In this section we provide a verification result (Theorem \ref{prop:verifPDE}) which addresses the question of existence of a Nash equilibrium in the agent-form game (i.e., equivalently of a saddle point for the ex-ante game)
from the point of view of PDE theory. In particular we show that a triple of functions $(u, u^0,u^1)$ with $u:=u_0+\varphi u_1$  that solves an appropriate quasi-variational inequality provides the equilibrium payoffs for the game as in \eqref{NE2}, \eqref{NE3} and \eqref{NE1}. This is done by identifying a Nash equilibrium from the candidate functions $(u, u^0,u^1)$.
The formulation in terms of a quasi-variational inequality bridges the probabilistic formulation of our problem to PDE theory and will be used in the next section to construct a full solution to a specific example with linear payoffs.

Denote by $W^{2,\infty}_{loc}(\cI\times(0,+\infty))$ the usual Sobolev space of functions in $L^\infty_{loc}$ whose first and second derivatives are also functions in $L^\infty_{loc}$ (recall also that letting $C^1_K$ be the space of $C^1$ functions on a compact $K$, by Sobolev embedding $W^{2,\infty}_{loc}\subset C^1_K$ for any compact $K$, \cite[Thm.~4.12]{Adams}).
In what follows, for $i=0,1$, denote by $\LL^i$ the second order differential operator associated with the dynamics of $(X,\Phi)$ under the measure $\P^i$, that is
\begin{align}
\notag
\LL^0:=& \tfrac{1}{2}\Big(\omega^2(x)\varphi^2 \partial_{\varphi\varphi}+\sigma^2(x)\partial_{xx}+2(\sigma\omega)(x)\varphi\partial_{x\varphi}\Big)+\mu_0(x)\partial_x\,,\\
\LL^1:=& \tfrac{1}{2}\Big(\omega^2(x)\varphi^2 \partial_{\varphi\varphi}+\sigma^2(x)\partial_{xx}+2(\sigma\omega)(x)\varphi\partial_{x\varphi}\Big) +\mu_1(x)\partial_x+\omega^2(x)\varphi\partial_\varphi\,.\notag
\end{align}

In the next theorem we will use the following localising sequences of stopping times: for a $C^1$ function $h$, let
\begin{align*}
I(h)_t:= \int_0^t \!\!\left(\sigma^2(X_s)(\partial_x h)^2(X_s,\Phi^*_{s-})+\omega^2(X_s)\Phi^*_{s-}(\partial_\varphi h)^2(X_s,\Phi^*_{s-})\right)ds,
\end{align*}
with $\Phi^*$ as in \eqref{eq:phi*}, then we set
\begin{align}\label{tau-n}
\tau_n(h):=\inf\left\{t\ge 0:I(h)_t\ge n\right\}\wedge n.
\end{align}
Before stating the theorem we also notice that given a set $U\subset \cI\times(0,+\infty)$, its closure should be understood relatively to $\cI\times(0,+\infty)$, in the sense that $\overline U$ does not include the boundary of the state-space, i.e.~$\overline U\cap \partial(\cI\times(0,+\infty))=\varnothing$.

\begin{theorem}[\textbf{Quasi-variational inequality}]\label{prop:verifPDE}
Let Assumption \ref{ass:fg} hold. Let $u, u^0, u^1:\cI\times(0,+\infty)\to [0,\infty)$ be
continuous functions with $u:=u^0+\varphi u^1$. Denote
\begin{align*}
&\C:=\{(x,\varphi)\in\cI\times(0,+\infty): u(x,\varphi)> (1+\varphi)f(x)\}\,,\\
&\C^i:=\{(x,\varphi)\in\cI\times(0,+\infty): u^i(x,\varphi)< g(x)\}\,,
\end{align*}
and $\cS:=(\cI\times(0,+\infty))\setminus\C$, $\cS^i:=(\cI\times(0,+\infty))\setminus\C^i$ for $i=0,1$.

For $i=0,1$, assume that
\[
u\in W^{2,\infty}_{loc}(\C^0\cap\C^1 )\cap C^1(\overline{\C^0\cap\C^1})\cap C^2(\overline{\C\cap\C^0\cap\C^1}),\quad\text{and}\quad u^i\in C^2(\overline{\C\cap\C^0\cap\C^1})
\]
%and
%\[
%u^i\in C^2(\overline{\C\cap\C^0\cap\C^1}),
%\]
and that $(u,u^0,u^1)$ solve the quasi-variational inequality
\begin{align}
\label{PDEu}&\max\{\LL^0 u (x,\varphi), (1+\varphi)f(x)-u(x,\varphi)\} = 0, \quad a.e.~(x,\varphi)\in\C^0\cap\C^1\,,\\
\label{PDEui}&\LL^i u^i (x,\varphi)=0, \qquad\qquad\text{for all $(x,\varphi)\in\C\cap\C^0\cap\C^1$ and for $i=0,1$},
%&\max\{-\LL^i u^i (x,\varphi), u^i(x,\varphi)-g(x)\} = 0, & (x,\varphi)\in\C\cap\C^{1-i}\,,
\end{align}
%in the a.e.~sense,
with the additional conditions $u^i\le g$ for $i=0,1$ and
\begin{align}
\label{cont0}&u^i(x,\varphi)=f(x),\quad\text{for $(x,\varphi)\in\cS$,}\\
\label{smooth0}&u^i_\varphi(x,\varphi)=0,\quad\text{for $(x,\varphi)\in \cS^0\cup\cS^1$}.
\end{align}

Assume also that there exists $\Gamma^*\in\cA^\theta$, with $\P^i(\Gamma^{*,0}_t<1)=1$ and $\P^i(\Gamma^{*,1}_t<1)=1$, for all $t\geq 0$ and $i=0,1$, such that, recalling \eqref{eq:phi*},
we have: $\P^0$ and $\P^1$-a.s.,
\begin{align}
\label{SK2}& \Delta\Gamma^{*,0}_t\cdot\Delta\Gamma^{*,1}_t = 0,\:\:\text{for all $t\ge 0$}, \\[+3pt]
\label{SK0}&(X_t,\Phi^*_t)\in \overline{\C^0\cap\C^1},\:\:\text{for all $t\ge 0$},\\[+3pt]
\label{SK1}&\left\{
\begin{array}{l}
\text{for $i=0,1$ and for all $t\ge 0$},\\
d\,\Gamma^{i,*}_t=\mathds{1}_{\{(X_t,\Phi^*_{t-})\in\cS^i\}}d\,\Gamma^{i,*}_t\:\:\text{and}\\[+4pt]
\int_{\Phi_{t-}^*}^{\Phi_{t}^*}\mathds{1}_{\{(X_t,z)\notin\cS^i\}}dz=0.
\end{array}
\right.
\end{align}
Moreover, assume that $\tau^*:=\inf\{t\ge 0\,:\,(X_t,\Phi^*_t)\notin\C\}$ is finite $\P$-a.s., and that the transversality conditions
\begin{align}
\label{tr2}\lim_{n\to+\infty}\E^i_{x,\varphi}\left[\mathds{1}_{\{\tau^*>\tau_n\}}u^i(X_{\tau_n},\Phi^*_{\tau_n})\right]=0,\quad\quad i=0,1,
\end{align}
hold for $\tau_n=\tau_n(u^i)$ and $\tau_n=\tau_n(u)$ as in \eqref{tau-n}, and for all
$(x,\varphi)\in\cI\times(0,+\infty)$.

Then, letting $\gamma^*_\theta\in\cA^\theta$ be the randomised stopping time generated by $\Gamma^*$, we have that $(\tau^*,\gamma^*_\theta)$ forms a Nash equilibrium in the agent-form game (i.e., a saddle point in the ex-ante game). Consequently, a value $v$ exists in the ex-ante game, and the equilibrium payoffs in the agent-form game are given by
\begin{align}\label{tr}
v=u(x,\varphi)=\JJ_{x,\varphi}(\tau^*,\gamma^*_\theta)\:\:\text{and}\:\: u^i(x,\varphi)=\cJ^i_{x,\varphi}(\tau^*,\gamma^*_\theta),\:\:\text{for $i=0,1$}.
\end{align}
\end{theorem}

\begin{proof}
We start by observing that under our assumptions the stopping times $\tau_n(u^i)$, $i=0,1$, and $\tau_n(u)$ are such that $\tau_n(u^i),\tau_n(u)\to\infty$ as $n\to\infty$, $\P^0$ and $\P^1$-a.s.~(for this result we need $\Gamma^{*,i}_t<1$ for all $t\ge 0$).

{\bf Optimality of $\tau^*$}. Let $\tau\in\cT$. Denote by $\{\tau_n\}_{n=1}^\infty$ the localizing sequence
of stopping times $\tau_n=\tau_n(u)$.
If $u$ were twice continuously differentiable in the whole space, using that $\LL^0u\le 0$ on $\C^0\cap\C^1$, applying It\^o's formula and then taking expectations would give
\begin{eqnarray}\label{v00}
\E^0_{x,\varphi}\left[(1-\Gamma^{*,0}_{\tau\wedge\tau_n})u(X_{\tau\wedge\tau_n},\Phi^*_{\tau\wedge\tau_n})\right]
&\le&\,u(x,\varphi)-\E^0_{x,\varphi}\left[\int_0^{\tau\wedge\tau_n} u(X_t,\Phi^*_{t-})d\,\Gamma^{*,0,c}_t\right]\\
&&\hspace{-20mm}+\E^0_{x,\varphi}\left[\int_0^{\tau\wedge\tau_n}u_\varphi(X_t,\Phi^*_{t-})
\left(\Phi^*_{t-}d\,\Gamma^{*,0,c}_t-\Phi_{t} d\,\Gamma^{*,1,c}_t\right)\right]\notag\\
&&\hspace{-20mm}+\E^0_{x,\varphi}\Big[\sum_{t\le \tau\wedge\tau_n}\big((1-\Gamma^{*,0}_t)u(X_t,\Phi^*_t)-(1-\Gamma^{*,0}_{t-})u(X_t,\Phi^*_{t-})\big)\Big],\notag
\end{eqnarray}
where $\Gamma^{*,i,c}$ denotes the continuous part of $\Gamma^{*,i}$, $i=0,1$.
Using a mollifying argument (see, e.g., \cite[Thm.~4.1, Ch.~VIII]{FS}), the inequality \eqref{v00} can be obtained
also for $u$ with the assumed regularity,
upon noticing that $(X,\Phi^*)$ only takes values in $\overline{\C^0\cap\C^1}$ as per \eqref{SK0}.

Since $u_\varphi(x,\varphi)=u^0_\varphi(x,\varphi)+\varphi u^1_\varphi(x,\varphi)+u^1(x,\varphi)$ and recalling \eqref{SK1} we see that \eqref{smooth0} implies
\begin{align}\label{v01}
&u_\varphi(X_t,\Phi^*_{t-})\left(\Phi^*_{t-}d\,\Gamma^{*,0,c}_t- \Phi_td\,\Gamma^{*,1,c}_t\right)\\
&=u^1(X_t,\Phi^*_{t-})\Phi^*_{t-} d\,\Gamma^{*,0,c}_t-g(X_t)\Phi_t  d\,\Gamma^{*,1,c}_t.\notag
\end{align}
Then combining the integrals with respect to the continuous parts of the increasing processes one finds
\begin{align}\label{v02}
\E^0_{x,\varphi}&\Big[\int_0^{\tau\wedge\tau_n}\!\!\!u_\varphi(X_t,\Phi^*_{t-})\!
\left(\Phi^*_{t-}d\,\Gamma^{*,0,c}_t-\Phi_t d\,\Gamma^{*,1,c}_t\right)
\!-\!\int_0^{\tau\wedge\tau_n}\!\!\!u(X_t,\Phi^*_{t-})d\,\Gamma^{*,0,c}_t\Big]\\
&=-\E^0_{x,\varphi}\Big[\!\int_0^{\tau\wedge\tau_n}\!g(X_t)(d\,\Gamma^{*,0,c}_t+\Phi_td\,\Gamma^{*,1,c}_t)\Big]\notag.
\end{align}

Next, we compute the contributions from jumps and recall \eqref{SK2}.
On the event $\{\Delta\Gamma^{*,0}_t>0\}$ we have, recalling \eqref{smooth0} and \eqref{SK1},
\begin{align*}
&u^0(X_t,\Phi^*_{t})=u^0(X_t,\Phi^*_{t-})=g(X_t)\notag\\
& u^1(X_t,\Phi^*_{t})=u^1(X_t,\Phi^*_{t-}).
\end{align*}
Consequently, using \eqref{eq:phi*} and that $\Delta\Gamma^{*,1}_t=0$, we get
\begin{align}\label{jjj0}
&(1-\Gamma^{*,0}_{t})u(X_t,\Phi^*_t)-(1-\Gamma^{*,0}_{t-})u(X_t,\Phi^*_{t-})\\
&=(1-\Gamma^{*,0}_{t})\big(g(X_t)+\Phi^*_{t} u^1(X_t,\Phi^*_{t})\big)\notag\\
&\quad-(1-\Gamma^{*,0}_{t-})\big(g(X_t)+\Phi^*_{t-} u^1(X_t,\Phi^*_{t-})\big)=-\Delta \Gamma^{*,0}_t g(X_t) .\notag
\end{align}
Similarly, on the event $\{\Delta\Gamma^{*,1}_t>0\}$ we have
\begin{equation}
\label{v04}
(1-\Gamma^{*,0}_{t})\Big(u(X_t,\Phi^*_t)-u(X_t,\Phi^*_{t-})\Big)=-\Delta\Gamma^{*,1}_t\Phi_tg(X_t).
\end{equation}

By combining \eqref{v00}, \eqref{v02}, \eqref{jjj0} and \eqref{v04} we obtain
\begin{equation}\label{v04.5}
\E^0_{x,\varphi}\,\left[(1-\Gamma^{*,0}_{\tau\wedge\tau_n})u(X_{\tau\wedge\tau_n},\Phi^*_{\tau\wedge\tau_n})\right]
\le u(x,\varphi)-\E^0_{x,\varphi}\left[\int_0^{\tau\wedge\tau_n}g(X_t)(d\,\Gamma^{*,0}_t+\Phi_t d\,\Gamma^{*,1}_t)\right],
\end{equation}
where we notice that the integral with respect to the increasing processes now includes the jump part as well.
Rearranging terms and using that $u(x,\varphi)\ge (1+\varphi)f(x)$ for $(x,\varphi)\in\C^0\cap\C^1$ we get
\begin{equation}
\label{v06}
u(x,\varphi)\ge \E^0_{x,\varphi}\Big[(1-\Gamma^{*,0}_{\tau\wedge\tau_n})f(X_{\tau\wedge\tau_n})(1+\Phi^*_{\tau\wedge\tau_n})
+\int_0^{\tau\wedge\tau_n}g(X_t)(d\,\Gamma^{*,0}_t+\Phi_t d\,\Gamma^{*,1}_t)\Big].
\end{equation}
Passing to the limit as $n\to\infty$ and using Fatou's lemma gives
\[
u(x,\varphi)\ge \sup_{\tau\in\cT}\JJ_{x,\varphi}(\tau,\gamma^*_\theta).
\]

To obtain the reverse inequality we repeat the steps above with $\tau^*\wedge\tau_n$ in place of $\tau$,
where $\tau_n=\tau_n(u)$ as in \eqref{tau-n}.
In this case we can use standard It\^o's formula because $u\in C^2(\overline{\C\cap\C^0\cap\C^1})$ and $(X_{t\wedge\tau^*},\Phi^*_{t\wedge\tau^*})_{t\ge 0}$ is bound to evolve in $\overline{\C\cap\C^0\cap\C^1}$.
Then the inequality in \eqref{v00} is an equality, so \eqref{v04.5} becomes
\begin{align*}
u(x,\varphi) \!&=\E^0_{x,\varphi}\!\left[(1\!-\!\Gamma^{*,0}_{\tau^*\wedge \tau_n})u(X_{\tau^*\wedge  \tau_n},\Phi^*_{\tau^*\wedge n})\!+\!\!\int_0^{\tau^*\wedge  \tau_n}\!\!\!\!g(X_t)(d\,\Gamma^{*,0}_t\!+\!\Phi_t d\,\Gamma^{*,1}_t)\right] \\
&= \,\E^0_{x,\varphi}\left[(1-\Gamma^{*,0}_{\tau^*})f(X_{\tau^*})(1+\Phi^*_{\tau^*})\mathds{1}_{\{\tau^*\leq  \tau_n\}} +(1-\Gamma^{*,0}_{ \tau_n})   u(X_{\tau_n},\Phi^*_{ \tau_n})\mathds{1}_{\{ \tau_n<\tau^*\}} \right]\\
&\quad\qquad+\,\E^0_{x,\varphi}\left[\int_0^{\tau^*\wedge  \tau_n}g(X_t)(d\,\Gamma^{*,0}_t+\Phi_t d\,\Gamma^{*,1}_t)\right],
\end{align*}
where we have used that $u(X_{\tau^*},\Phi^*_{\tau^*})=f(X_{\tau^*})(1+\Phi^*_{\tau^*})$.
From $u(x,\varphi)=u^0(x,\varphi)+\varphi u^1(x,\varphi)$ and \eqref{tr2} we obtain
\[\lim_{n\to+\infty}\E^0_{x,\varphi}\left[(1-\Gamma^{*,0}_{\tau_n})u(X_{ \tau_n},\Phi^*_{ \tau_n})\mathds{1}_{\{ \tau_n<\tau^*\}}\right]=0,\]
upon recalling the change of measure \eqref{eq:P1P0}. So using monotone convergence we take limits as $n\to \infty$ to conclude that
\[u(x,\varphi)= \sup_{\tau\in\cT}\JJ_{x,\varphi}(\tau,\gamma^*_\theta) =\JJ_{x,\varphi}(\tau^*,\gamma^*_\theta). \]

{\bf Optimality of $\Gamma^*$}.
Pick $\Gamma\in\cA^\theta$ and note that $(X_{t\wedge\tau^*},\Phi^*_{t\wedge\tau^*})_{t\ge 0}\in \overline{\C\cap\C^0\cap\C^1}$.
Since $u^i\in C^2(\overline{\C\cap\C^0\cap\C^1})$ for $i=0,1$, we can apply standard It\^o's formula to $u^i(X,\Phi^*)$
and use that $\LL^i u^i= 0$~on $\C\cap\C^0\cap\C^1$. This gives
\begin{eqnarray}\label{v07}
\E^i_{x,\varphi}\left[(1-\Gamma^{i}_{\tau^*\wedge\tau_n})u^i(X_{\tau^*\wedge\tau_n},\Phi^*_{\tau^*\wedge\tau_n})\right]
&=&\,u^i(x,\varphi)-\E^i_{x,\varphi}\left[\int_0^{\tau^*\wedge\tau_n}u^i(X_t,\Phi^*_{t-})d\,\Gamma^{i,c}_t\right]\\
&&\hspace{-30mm} +\E^i_{x,\varphi}\left[\int_0^{\tau^*\wedge\tau_n}\frac{1-\Gamma_{t-}^i}{1-\Gamma_{t-}^{*,0}}\:\:u^i_\varphi(X_t,\Phi^*_{t-})
(\Phi^*_{t-} d\,\Gamma^{*,0,c}_t- \Phi_td\,\Gamma^{*,1,c}_t)\right]\notag\\
&&\hspace{-30mm}+\E^i_{x,\varphi}\Big[\sum_{t\le \tau^*\wedge\tau_n}\big((1-\Gamma^{i}_t)u^i(X_t,\Phi^*_t)-(1-\Gamma^{i}_{t-})u^i(X_t,\Phi^*_{t-})\big)\Big], \notag
\end{eqnarray}
where $\{\tau_n\}_{n=1}^\infty$ is the localizing sequence of stopping times $\tau_n=\tau_n(u^i)$.
Recalling that $u^i_\varphi=0$ on the support of $t\mapsto d\,\Gamma^{0,*}_t$ and $t\mapsto d\,\Gamma^{1,*}_t$ (cf.~\eqref{smooth0} and \eqref{SK1}) we immediately see that
\begin{align}\label{v05}
\E^i_{x,\varphi}\left[\int_0^{\tau^*\wedge\tau_n}
\frac{1-\Gamma_{t-}^i}{1-\Gamma_{t-}^{*,0}}\:\:u^i_\varphi(X_t,\Phi^*_{t-})
(\Phi^*_{t-} d\,\Gamma^{*,0,c}_t- \Phi_td\,\Gamma^{*,1,c}_t)\right]=0.
\end{align}
Moreover, \eqref{SK1} guarantees
\begin{align*}
u^i(X_t,\Phi^*_t)-u^i(X_t,\Phi^*_{t-})=0,\qquad\text{$\P^i_{x,\varphi}$-a.s.}
\end{align*}
so that by simply adding and subtracting $(1-\Gamma^i_{t-})u^i(X_t,\Phi^*_t)$ in the sum of jumps in \eqref{v07} we obtain
\begin{align}\label{v08}
&\E^i_{x,\varphi}\Big[\sum_{t\le \tau^*\wedge\tau_n}\big((1-\Gamma^{i}_t)u^i(X_t,\Phi^*_t)-(1-\Gamma^{i}_{t-})u^i(X_t,\Phi^*_{t-})\big)\Big]\\
&=-\E^i_{x,\varphi}\Big[\sum_{t\le \tau^*\wedge\tau_n}u^i(X_t,\Phi^*_t)\Delta \Gamma^{i}_t\Big]\notag\\
&\ge- \E^i_{x,\varphi}\Big[\sum_{t\le \tau^*\wedge\tau_n}\!g(X_t)\Delta \Gamma^{i}_t\Big],\notag
\end{align}
where the final inequality uses that $u^i\le g$ on $\C\cap\C^{1-i}$.

Next, plugging \eqref{v05} and \eqref{v08} in \eqref{v07}, and using again that $u^i\le g$ on $\C\cap\C^{1-i}$, we arrive at
\begin{align}\label{v09}
\E^i_{x,\varphi}\left[(1-\Gamma^{i}_{\tau^*\wedge\tau_n})u^i(X_{\tau^*\wedge\tau_n},\Phi^*_{\tau^*\wedge\tau_n})\right]
\ge u^i(x,\varphi)-\E^i_{x,\varphi}\left[\int_0^{\tau^*\wedge\tau_n}\!\! g(X_t)d\,\Gamma^{i}_t\right],
\end{align}
where the integral now includes both the continuous part and the jump part of the increasing process.
Using \eqref{cont0} we see that
\begin{align}\notag
u^i(x,\varphi) \leq& \,\E^i_{x,\varphi}\left[(1-\Gamma^{i}_{\tau^*})f(X_{\tau^*})\mathds 1_{\{\tau^*\leq \tau_n\}}\right]\\
&+\,\E^i_{x,\varphi}\left[(1-\Gamma^{i}_{\tau_n})u^i(X_{\tau_n},\Phi^*_{\tau_n})\mathds 1_{\{\tau_n<\tau^*\}}\right]\!\! +\E^i_{x,\varphi}\left[\int_0^{\tau^*\wedge\tau_n}\!\!g(X_t)d\,\Gamma^{i}_t\right].\notag
\end{align}
Passing to the limit as $n\to\infty$, using the transversality condition \eqref{tr2}, monotone convergence and \eqref{eq:P1P0} we obtain
\begin{align}
u^i(x,\varphi)\le \E^i_{x,\varphi}\left[(1-\Gamma^{i}_{\tau^*})f(X_{\tau^*})+\int_0^{\tau^*}g(X_t)d\,\Gamma^{i}_t\right].\notag
\end{align}
Consequently,
\[
u^i(x,\varphi)\le \inf_{\Gamma\in \cA^\theta}\mathcal J^i_{x,\varphi}(\tau^*,\gamma_\theta),\quad\text{for $i=0,1$}.
\]

The reverse inequality is obtained by taking $\Gamma=\Gamma^{*}$ in the proof above and observing that in doing so the inequalities
in \eqref{v08} and \eqref{v09} become equalities. We thus obtain
\[
u^i(x,\varphi)= \inf_{\Gamma\in \cA^\theta}\mathcal J^i_{x,\varphi}(\tau^*,\gamma_\theta)= \mathcal J^i_{x,\varphi}(\tau^*,\gamma^*_\theta),\]
for $i=0,1$, which completes the proof.
\end{proof}

\begin{figure}[htp!]\centering
  %\psfrag{Phi}{$\Phi$}
  %\psfrag{X}{$X$}
  %\psfrag{S0}{$\cS^0$}
  %\psfrag{S1}{$\cS^1$}
  %\psfrag{tau}{$\tau^*$}
  %\psfrag{S}{$\cS$}
  %\psfrag{d1}{$d\Gamma^{*,1}$}
  %\psfrag{d0}{$d\Gamma^{*,0}$}
  %\psfrag{map}{$t\mapsto(X_t,\Phi_t^*)$}
  %\psfrag{C}{$\C\cap\C^0\cap\C^1$}
%  \includegraphics[width=0.75\textwidth]{Schematic3.eps}
 \includegraphics[width=0.75\textwidth]{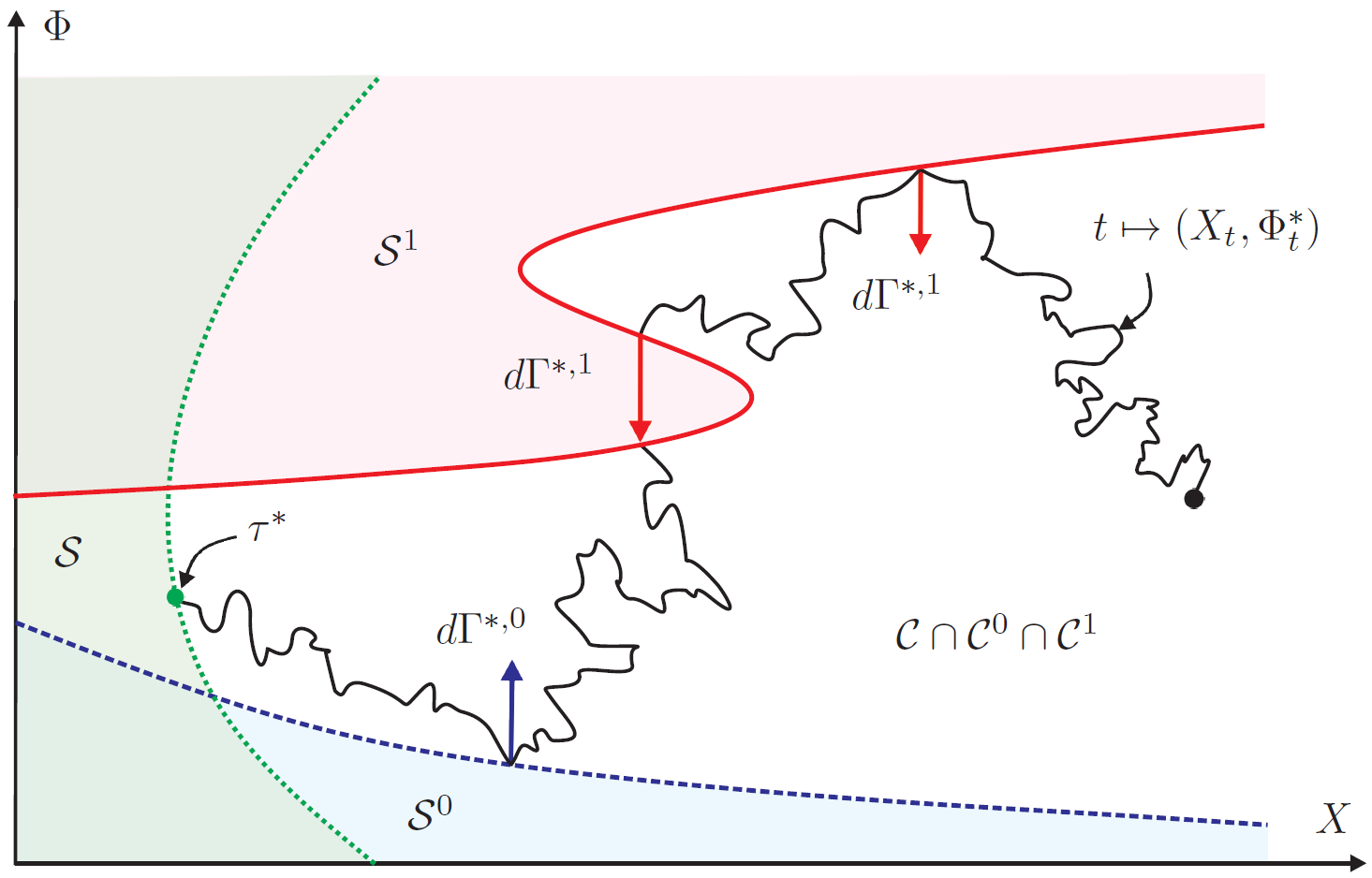}
  \caption{A schematic picture of the different regions in the $(x,\varphi)$-plane. The controls $\Gamma^{*,i}$ are constructed so that the process $(X,\Phi^*)$ reflects at $\partial(\C\cap\C^0\cap\C^1)$. Note that, depending on the geometry of $\C\cap\C^0\cap\C^1$,
the process may be discontinuous.} \label{fig:schematic}
\end{figure}

\begin{remark}
The assumption $u\in C^2(\overline{\C\cap\C^0\cap\C^1})$ is needed in the generality of the theorem because a priori the law of $(X,\Phi^*)$ may have atoms on the boundary of the domain, i.e.~on $\partial(\overline{\C\cap\C^0\cap\C^1})$. However, in practical examples where something is known about the geometry of $\C^0\cap\C^1$ one may be able to rule out the existence of such atoms and the assumption may be relaxed to $u\in C^1(\overline{\C\cap\C^0\cap\C^1})$ with bounded second derivatives.
\end{remark}

\begin{remark}
The assumption that $\P^i(\Gamma^{*,0}_t<1)=1$ and $\P^i(\Gamma^{*,1}_t<1)=1$ for all $t\ge 0$ is useful for the localisation of the stochastic integrals in the proof and to avoid that the process $\Phi^*$ reaches the endpoints of its state-space (where $u$ and $u^i$ are not properly defined). It also has a natural interpretation in those games where Player 2 does not want to fully reveal her informational advantage at any finite time but instead where a gradual release of information is optimal (as we observe in Section \ref{sec:linear}). Indeed, if for example $\P^i(\Gamma^{*,0}_t=1)>0$ for some $t\ge 0$, then full information is revealed at time $t$ for all $\omega\in\{\Gamma^{*,0}_t=1\}$.
\end{remark}

For a schematic illustration of the reflection of the process $(X,\Phi^*)$, see Figure~\ref{fig:schematic}.
We are not aware of any standard PDE results that guarantee the solvability of the quasi-variational inequality above. Nevertheless, the structure of \eqref{PDEu}--\eqref{PDEui} resembles that of quasi-variational inequalities for nonzero-sum Dynkin games (see, e.g., \cite{BF} and more recently \cite{DeAFM}), as we should expect from Proposition \ref{prop:NE} and Remark \ref{rem:nzsg}. Hence one may hope that general existence of solutions can be found following ideas from that literature.

We will show in the next section that the assumptions in Theorem \ref{prop:verifPDE} hold in an example with a linear payoff structure.

\section{An example with linear payoffs}\label{sec:linear}

In this section we study an example where the underlying diffusion is a geometric Brownian motion and the payoff functions are linear.
It is hoped that our analysis in this specific example can be used to inform future work on more general optimal stopping games and on the solvability of the quasi-variational inequality that we derived in Section~\ref{sec:verif}.

To describe the example, let
\[dX_t=\mu X_t\,dt + \sigma X_t\,dW_t,\]
where $\mu=\mu_0(1-\theta)+\mu_1\theta$ and (with a small abuse of notation) $\mu_0$ and $\mu_1$ now are constants satisfying $\mu_0<\mu_1$.
In this case, the signal-to-noise ratio $\omega=(\mu_1-\mu_0)/\sigma$ is also a constant.
Furthermore, let
\begin{align}\label{linp}
f(x)=x \quad \textrm{and} \quad g(x)=(1+\ep) x,
\end{align}
where $\ep>0$.
Given a randomised stopping pair $(\tau,\gamma_\theta)\in\cT\times\cT^\theta_R$,
the stopping game with asymmetric information has a payoff
\[R(\tau,\gamma_\theta)=X_{\tau}1_{\{\tau<\gamma_\theta\}} + (1+\ep)X_{\gamma_\theta}1_{\{\tau\geq\gamma_\theta\}} ,\]
where we also recall that under $\P^0$ we have
\begin{eqnarray*}
\left\{
\begin{array}{l}
dX_t=\mu_0 X_tdt+\sigma X_t d\WO_t\,,\\[+3pt]
d\Phi_t=\omega\Phi_t d\WO_t\,,
\end{array}
\right.
\end{eqnarray*}
and under $\P^1$ we have
\begin{eqnarray*}
\left\{
\begin{array}{l}
dX_t=\mu_1 X_tdt+\sigma X_t d\WI_t\,,\\[+3pt]
d\Phi_t=\omega^2\Phi_t dt+\omega\Phi_t d\WI_t\,.
\end{array}
\right.
\end{eqnarray*}

\begin{remark}\label{rem5.1}
Clearly, the case $\mu=\mu_0$ is advantageous for Player~2, while the case $\mu=\mu_1$ would be preferred by Player~1.
Furthermore, if $\mu_0<\mu_1<0$, then the inf-player (Player~2) would never stop, whereas if $0<\mu_0<\mu_1$ the sup-player (Player 1) would never stop.
\end{remark}

In light of the above remark, in the rest of this section we make the following standing assumption.

\begin{assumption}\label{ass:mus}
We have $\mu_0<0<\mu_1$.
\end{assumption}

One key advantage of the linear structure of our example is that we can effectively reduce the problem to one state variable, hence simplifying the rest of the analysis. In particular we will see below that $\Phi$ is the only relevant dynamic in the optimisation. For $i=0,1$ let $\PP^i$ be defined by
\begin{align}\label{Ptilde}
\frac{d\PP^i}{d\P^i}\Big|_{\mathcal F^X_t}=\exp\left\{  -\frac{\sigma^2}{2}t +\sigma W^i_t \right\}\,,
\end{align}
and notice that $\W^i_t:=-\sigma t + W^i_t$ is a $\PP^i$-Brownian motion.
For future reference we note that
\begin{align}\label{Xtilde}
dX_t=(\mu_i+\sigma^2)X_t dt+\sigma X_t d\W^i_t\,,\quad\text{under $\PP^i$}
\end{align}
and
\begin{align}
\label{Phi1} &d\Phi_t=\sigma\omega\Phi_t dt+\omega\Phi_td\W^0_t,\quad\qquad\quad\text{under $\PP^0$},\\
\label{Phi2} &d\Phi_t=(\sigma\omega+\omega^2)\Phi_t dt+\omega\Phi_td\W^1_t,\quad\text{under $\PP^1$}.
\end{align}
It is also easy to verify that $\Phi$ and $X$ are effectively linked by direct proportionality, that is
\begin{align}\label{ratio}
\varphi^{-1}\Phi_t =x^{-\omega/\sigma}\big(X_t\big)^{\omega/\sigma}e^{\big((i-1/2)\omega^2-\frac{\omega}{\sigma}(\mu_i-\tfrac{\sigma^2}{2})\big)t}\,,\quad\text{$\PP^i$-a.s.}
\end{align}

\begin{lemma}\label{lem:1D}
For $(x,\varphi)\in\R_+\times\R_+$ and $(\tau,\gamma_\theta)\in\cT\times\cT^\theta_R$ our game payoff can be rewritten as
\begin{align}\label{JJlast}
\JJ_{x,\varphi}(\tau,\gamma_\theta) =
x \Big(&\EE^0_{\varphi}\Big[ e^{\mu_0 \tau} (1-\Gamma^0_\tau)+ (1+\ep)\!\int_0^\tau \!e^{\mu_0 t} d\Gamma^0_t\Big]\\
& + \EE^0_{\varphi}\Big[e^{\mu_0 \tau} \Phi_\tau (1-\Gamma^1_\tau)  + (1+\ep)\!\int_0^\tau\! e^{\mu_0 t} \Phi_t   d\Gamma^1_t\Big]\Big).\notag
\end{align}
Moreover, we have $\JJ_{x,\varphi}(\tau,\gamma_\theta)=\cJ^0_{x,\varphi}(\tau,\Gamma^0)+\varphi\cJ^1_{x,\varphi}(\tau,\Gamma^1)$, where
\begin{align}\label{J0last}
\cJ^0_{x,\varphi}(\tau,\Gamma^0) = x\, \EE^0_{\varphi}\Big[e^{\mu_0\tau}(1-\Gamma^0_\tau) + (1+\ep)\!\int_0^\tau\! e^{\mu_0 t}d\Gamma^0_t\Big]
\end{align}
and
\begin{align}\label{J1last}
\cJ^1_{x,\varphi}(\tau,\Gamma^1) = x\,\EE^1_{\varphi}\Big[e^{\mu_1\tau}(1-\Gamma^1_\tau) + (1+\ep)\!\int_0^\tau\! e^{\mu_1 t}d\Gamma^1_t\Big].
\end{align}
\end{lemma}

\begin{proof}
The expression in \eqref{JJlast} follows from \eqref{JJ1} and \eqref{JJ}, upon noticing that
\[
X_t= x\, e^{\mu_0 t}\frac{d\PP^0}{d\P^0}\Big|_{\cF^X_t},\qquad \P^0\textrm{-a.s.},
\]
and arguing as in the proof of Corollary \ref{game-reduced}.
Likewise, \eqref{J0last} and \eqref{J1last} follow from \eqref{J0} and \eqref{J1}.
\end{proof}

It is intuitively clear that Player 2 should never stop in the case $\mu=\mu_0$.
Note that, for any $\Gamma^0\in\cA$, integration by parts allows us to rewrite \eqref{J0last} as
\begin{align*}
\cJ^0_{x,\varphi}(\tau,\Gamma^0)=&\,x\Big(1+\EE^0_{\varphi}\Big[\mu_0\int_0^{\tau}e^{\mu_0 t}(1 -\Gamma^0_t) dt+\ep\int_0^{\tau}e^{\mu_0 t} d\,\Gamma^0_t\Big]\Big).
\end{align*}
Using that $\mu_0<0$, we immediately obtain
\[
\cJ^0_{x,\varphi}(\tau,\Gamma^0)\geq x\Big(1+\EE^0_{\varphi}\Big[\mu_0\int_0^{\tau}e^{\mu_0 t} dt\Big]\Big)=\cJ^0_{x,\varphi}(\tau,0),\]
i.e. \eqref{cond} holds. Consequently, Proposition~\ref{prop:G0} shows that it is sufficient to look for a Nash equilibrium
in the subclass of $\gamma\in\cT^\theta_R$ for which $\gamma_0=+\infty$ (equivalently, $\Gamma^0=0$).

Now we need to work out the remaining equilibrium control $\Gamma^{*,1}$ and the stopping time $\tau^*$.
We first formulate an educated guess on the structure of $\tau^*$ and $\Gamma^{*,1}$, and
subsequently we verify that using such a guess we can produce a solution of the
quasi-variational inequality from Theorem~\ref{prop:verifPDE}.

\subsection{Candidate adjusted likelihood ratio}

In the case $\mu=\mu_1$, the existence of asymmetric information creates an incentive for the informed player not to stop immediately  in order to `fool' the uninformed player. Indeed, if the uninformed player is made to believe that the drift is low (i.e., $\mu=\mu_0$),
then the uninformed player may choose to stop early, which is beneficial for the informed player as only the smaller payoff
has to be paid. Thus it is natural that Player 2 will only want to stop when $\Phi$ becomes too high (i.e., the uninformed player has a strong belief that the drift is $\mu_1$).

Including the idea of randomisation in the reasoning above, we expect that the informed player will stop at some upper threshold according to some `intensity'.
The effect of randomisation is to generate an adjusted likelihood ratio $\Phi^*$,
which can be interpreted as the belief of the uninformed player after manipulation performed by the informed one.
For Player 2 it is therefore a question of finding the optimal trade-off between manipulating Player 1's beliefs and stopping not too late.

Following the heuristics above we conjecture that Player 2 will construct $\Gamma^{*,1}$ in a way that reflects the process
$\Phi^*=\Phi(1-\Gamma^{*,1})$ at an upper threshold.
With this idea in mind, let $B\in(0,\infty)$ and an initial belief $\varphi\in(0,\infty)$ be given.
It is well known that there exists a unique pair of processes $(Y,L)$ such that $\PP^0_\varphi$-a.s.~one has
\begin{align}
\label{SDEr0}&\text{$(L)_{t\geq 0}$ is continuous and non-decreasing with $L_{0}=0$},\\
\label{SDEr1}&\text{$Y_{0-}=\varphi$, $Y_{0}=\varphi\wedge B$ and $Y_t\in(0,B]$ for $t\geq 0$,}\\
\label{SDEr}&\text{$(Y,L)$ solves}\quad\left\{\begin{array}{ll}
dY_t=\sigma\omega Y_t\,dt + \omega Y_t\,d\W^0_t -dL_t\,,\\
\int_0^t1_{\{Y_s<B\}}dL_s=0.
\end{array}\right.
\end{align}
Then $Y$ is a diffusion process with reflection at $B$.
Define the process $\Gamma^{B}\in\cA$ by $\Gamma^{B}_{0-}=0$, $\Gamma_0^{B}=\max\{0,1-B/\varphi\}$ and
\begin{align}\label{gammaB}
\Gamma^{B}_t=1-(1-\Gamma^{B}_0)e^{-L_t/B}\,,\quad\PP^0_\varphi\textrm{-a.s.}
\end{align}
Next we show that the adjusted likelihood ratio corresponding to the pair  $\Gamma=(0,\Gamma^{B})$ is given by the
reflected process $Y$.

\begin{proposition}
\label{result}
Fix $B\in(0,\infty)$, and consider the processes $(Y,L)$ and $\Gamma^{B}$ as above.
Then for any $\varphi\in(0,\infty)$ we have
\begin{align}\label{refl-Phi}
\Phi^B_t:=\Phi_t (1-\Gamma^{B}_t)=Y_t,\quad\text{for all $t\ge 0$, $\PP^0_\varphi$-a.s.}
\end{align}
\end{proposition}

\begin{proof}
Noticing that \eqref{SDEr} implies $dL_t=\mathds{1}_{\{Y_t=B\}}dL_t$ we can write the first equation in \eqref{SDEr} as
\begin{align*}
dY_t=\sigma\omega Y_t\,dt + \omega Y_t\,d\W^0_t -B^{-1}Y_t dL_t.
\end{align*}
Recalling now \eqref{SDEr0}--\eqref{SDEr1} and thanks to the above equation we can write $Y$ explicitly under $\PP^0_\varphi$ as
\begin{align*}
Y_t=&(\varphi\wedge B) \exp\left(\omega \W^0_t+(\sigma\omega-\tfrac{\omega^2}{2})t-B^{-1}L_t\right).
\end{align*}
A direct comparison of the expression above with $\Gamma^{B}$ in \eqref{gammaB} and $\Phi$ in \eqref{Phi1} gives \eqref{refl-Phi}.
\end{proof}

Below we formulate and solve a variational problem based on the conjecture mentioned at the beginning of the section:
Player 2 will select a threshold $B\in\R_+$ and adopt the randomised stopping time generated by the couple $(0,\Gamma^{B})\in\cA\times\cA$.
Player 1 will instead choose a threshold $A\in(0,B)$ and stop at
\begin{align}\label{tauA}
\tau_A:=\inf\{t\ge 0\,:\, \Phi^B_t\le A\}.
\end{align}

\subsection{Quasi-variational inequality for the problem with linear payoff}

Here we use a constructive approach to obtain the candidate solution to the quasi-variational inequality for the game, which we will then test against the requirements of Theorem \ref{prop:verifPDE} in the next section.

As mentioned above, we look for an equilibrium with $\Gamma^{*,0}\equiv 0$.
If Player~2 plays $\Gamma=(0,\Gamma^{B})$ and Player~1 plays $\tau_A$, we obtain from \eqref{J0last}
\begin{align}\label{xV0}
\mathcal J^0_{x,\varphi}(\tau_A,\Gamma^{*,0}) =x\EE^0_\varphi\left[e^{\mu_0 \tau_A} \right]=:xV_0(\varphi).
\end{align}
The idea is that we should verify that $u^0(x,\varphi)=xV_0(\varphi)$ with $u^0$ as in Theorem~\ref{prop:verifPDE}.
This will be done in Theorem \ref{prop:final}, using facts collected in this section.

It is easy to check (see, e.g., \cite{SLG}) that $V_0$ satisfies
\begin{equation}
\label{fbp1}
\left\{\begin{array}{ll}
\frac{\omega^2\varphi^2}{2}V_0''(\varphi) + \sigma\omega \varphi V_0'(\varphi) + \mu_0 V_0(\varphi)=0, & \varphi\in(A,B)\\
V_0(\varphi)=1, & \varphi\in(0,A]\\
V_0'(B-)=0. & \end{array}\right.
\end{equation}
Notice that the condition at $B$ is the usual normal reflection condition.
Moreover, observing that $\PP^0_\varphi(\Phi_0^B=B)=1$ for $\varphi\geq B$,
it follows that $\EE^0_\varphi\left[e^{\mu_0 \tau_A} \right]=\EE^0_B\left[e^{\mu_0 \tau_A} \right]$ for $\varphi\ge B$, and therefore
\begin{align}\label{V0-B}
V'_0(\varphi)=0,\qquad\varphi\ge B.
\end{align}
Moreover, we also note that $\mu_0<0$ implies that
\begin{align}\label{V0le1}
V_0(\varphi)\le 1,\qquad\varphi\ge A.
\end{align}

Conversely, an application of It\^o's formula gives
\begin{lemma}\label{lem:V0}
Assume there exists $\bar V_0\in C^1([A,+\infty))\cap C^2([A,B])$ that solves \eqref{fbp1}--\eqref{V0-B}. Then $\bar V_0(\varphi)=\EE^0_\varphi\big[e^{\mu_0 \tau_A}\big]=V_0(\varphi)$.
\end{lemma}

Next we introduce a function $xV_1(\varphi)$ which we want to associate with $\cJ^1_{x,\varphi}(\tau_A,\Gamma^{B})$ from \eqref{J1last}. We cast a boundary-value problem for $V_1$ according to the following logic:
\begin{itemize}
\item[(i)] In the interval $(A,B)$ neither of the two players should stop, so the function $V_1$ should be harmonic for the process $\Phi^B$ with creation at rate $\mu_1$;
\item[(ii)] The informed player will only stop when the process $\Phi^B$ exceeds $B$ (although not necessarily at the first hitting time of $B$). Then we expect $V_1(B)=1+\ep$;
\item[(iii)] For the choice of $B$ to be optimal for Player 2 (given that Player~1 uses $\tau_A$), the classical smooth-fit condition should hold, that is we expect $V'_1(B-)=0$;
\item[(iv)] If the uninformed player stops first (according to $\tau_A$) then the cost for Player 2 is $V_1(A)=1$.
\end{itemize}
Combining the four items above gives us the boundary value problem
\begin{equation}
\label{fbp2}
\left\{\begin{array}{ll}
\frac{\omega^2\varphi^2}{2}V_1''(\varphi) + (\omega^2+\sigma\omega) \varphi V_1'(\varphi) + \mu_1 V_1(\varphi)=0, & \varphi\in(A,B)\\
V_1(\varphi)=1, & \varphi\in(0,A]\\
V_1(\varphi)=1+\ep, & \varphi\in[B,\infty)\\
V_1'(B)=0. & \end{array}\right.
\end{equation}

Notice that if $V_1\in C^1([A,+\infty))\cap C^2([A,B]) $ solves the above system, then it is easy to verify
\begin{align}\label{xV1}
xV_1(\varphi)=\cJ^1_{x,\varphi}(\tau_A,\Gamma^{B})
\end{align}
thanks to an application of It\^o calculus. Further we can establish monotonicity of $V_1$, which will be useful later in this section.

\begin{lemma}\label{lem:V1}
Assume that $0<A<B$ and  $V_1\in C^2([A,B])$ solves \eqref{fbp2}. Then $V_1'(\varphi)\ge 0$ for $\varphi\in[A,B]$ and $1\le V_1(\varphi)<1+\ep$ for $\varphi\in[A,B)$.
\end{lemma}

\begin{proof}
Let us start by observing that the first, third and fourth equations in \eqref{fbp2} imply
\[
\tfrac{\omega^2B^2}{2}V''_1(B-)=-\mu_1(1+\ep)<0.
\]
The above and $V'_1(B)=0$ imply that there exists $\lambda_0>0$ (with $B-\lambda_0\ge A$) such that
\begin{align}\label{V1prime}
V'_1(\varphi)>0\,,\qquad\text{for $\varphi\in(B-\lambda_0,B)$}.
\end{align}

With the aim of reaching a contradiction, assume that there exists $\varphi\in(A,B)$ such that $V'_1(\varphi)<0$. Then we can also define
\[
c:=\sup\{\varphi \in (A,B)\,:\,V'_1(\varphi)<0\},
\]
and clearly $c\in (A,B-\lambda_0]$. Due to continuity of $V'_1$ it must be $V'_1(c)=0$.
Since the ODE is of Euler type, its solution is a linear combination of power functions; thus
$V_1\in C^\infty(A,B)$.
Then, setting $v_1:=V'_1$ and differentiating the first equation in \eqref{fbp2} we get that $v_1$ must solve the boundary value problem
\begin{equation}
\label{bp0}
\left\{\begin{array}{ll}
\frac{\omega^2\varphi^2}{2}v_1''(\varphi) + (2\omega^2+\sigma\omega) \varphi v_1'(\varphi) + \alpha v_1(\varphi)=0, & \varphi\in(c,B),\\
v_1(c+)=v_1(B-)=0, & \end{array}\right.
\end{equation}
with $\alpha:=\omega^2+\sigma\omega+\mu_1$.
It then follows, e.g.~by the Feynman-Kac formula, that $v_1(\varphi)=0$ for $\varphi\in(c,B)$, which contradicts \eqref{V1prime}.

Then $V'_1(\varphi)\ge 0$ in $[A,B]$ as claimed. Moreover, $1\le V_1(\varphi)<1+\ep$ for $\varphi\in[A,B)$, by the second and third equation in \eqref{fbp2}.
\end{proof}

Hereafter, when referring to $V_1$ we will implicitly assume that it solves \eqref{fbp2} (we show in the next
subsection that \eqref{fbp2} and \eqref{fbp4} below can be solved simultaneously in a unique way).

Recalling \eqref{JJlast} it is now natural to associate $\JJ_{x,\varphi}(\tau_A,0,\Gamma^{1,B})$ to the function
\[
x V(\varphi):=x (V_0(\varphi)+\varphi V_1(\varphi)).
\]
Using the second equations in \eqref{fbp1} and \eqref{fbp2}, we see immediately that $V(A)=1+A$. Moreover, recalling also \eqref{JJ}, the function $\Xi(x,\varphi):=x V(\varphi)/(1+\varphi)$ should represent the equilibrium payoff for the uninformed player in the agent-form game
(notice that indeed $\Xi(x,A) = x$). Thus, by optimality, we expect that the classical smooth-fit condition holds at the boundary $A$, i.e.~$\Xi_\varphi(x,A+)=0$ or, equivalently,~$V'(A+)=1$. Hence, using \eqref{fbp1} and \eqref{fbp2} we obtain that $V$ should solve the boundary value problem
\begin{equation}
\label{fbp4}
\left\{\begin{array}{ll}
\frac{\omega^2\varphi^2}{2} V''(\varphi) + \sigma\omega \varphi  V'(\varphi) + \mu_0 V(\varphi) =0, &\varphi\in(A,B)\\
V(\varphi)=1+ \varphi, & \varphi\in(0,A]\\
V'(A+)=1,& \\
V'(B-)=1+\ep. &
\end{array}\right.
\end{equation}
Moreover,
\begin{align}\label{eq:Vp}
V'(\varphi)=1+\ep,\qquad\varphi\ge B
\end{align}
by \eqref{V0-B} and \eqref{fbp2}.
Before closing this section we provide some useful properties of $V$.

\begin{lemma}\label{lem:V}
Assume that $V\in C^2([A,B])$ solves \eqref{fbp4}. Then $V(\varphi)> 1+\varphi$ and $V'(\varphi)>1$ for $\varphi\in(A,B)$.
\end{lemma}

\begin{proof}
First note that the ODE is of Euler type, so $V\in C^\infty(A,B)$. Differentiating the first equation in \eqref{fbp4} and imposing the boundary conditions for $V'$ at $A$ and $B$, we find that $v:=V'$ solves
\begin{equation}
\label{bp}
\left\{\begin{array}{ll}
\frac{\omega^2\varphi^2}{2} v''(\varphi) + (\omega^2+\sigma\omega ) \varphi  v'(\varphi) + \mu_1 v(\varphi) =0, &\varphi\in(A,B)\\
v(A+)=1,& \\
v(B-)=1+\ep. &
\end{array}\right.
\end{equation}
Recalling the dynamics of $\Phi$ under $\PP^1$ (see \eqref{Phi2}), setting
\[
\rho_A:=\inf\{t\ge0\,:\,\Phi_t\le A\}\quad\text{and}\quad\rho_B:=\inf\{t\ge0\,:\,\Phi_t\ge B\},
\]
and using It\^o's formula, we easily obtain
\begin{align}\label{vge0}
v(\varphi)=&\EE^1_\varphi\Big[e^{\mu_1(\rho_A\wedge\rho_B)}v(\Phi_{\rho_A\wedge\rho_B})\Big]\\
=&\EE^1_\varphi\Big[e^{\mu_1(\rho_A\wedge\rho_B)}\Big]+\ep\,\EE^1_\varphi\Big[e^{\mu_1 \rho_B}\mathds{1}_{\{\rho_B<\rho_A\}}\Big].\notag
\end{align}

Now both claims in the lemma follow from \eqref{vge0}, due to $\mu_1>0$ and recalling that $V(\varphi)=1+\varphi$ for $\varphi\le A$.
\end{proof}

Finally, we notice that Lemma \ref{lem:V} and the fact that $V'(A+)=1$ imply that $V''(A+)\ge 0$. Then plugging the second and third equation of \eqref{fbp4} into the first one and using $V''(A+)\ge 0$ we obtain
\[
\frac{\omega^2 A^2}{2} V''(A+) + \sigma\omega A + \mu_0 (1+A) =0 \implies \mu_1 A+\mu_0 \le 0
\]
and the next corollary holds.
\begin{corollary}\label{cor:A}
Assume that $V\in C^2([A,B])$ solves \eqref{fbp4}. Then it must be $A\le-\mu_0/\mu_1$.
\end{corollary}

\subsection{Solution of the variational problem and Nash equilibrium}

We now show that \eqref{fbp2} and \eqref{fbp4} can be solved simultaneously in a unique way. Note that once
the functions $V$ and $V_1$ and the boundary points $A$ and $B$ are found, the function $V_0$ is automatically determined from the relation $V_0(\varphi):=V(\varphi)-\varphi V_1(\varphi)$.

The general solution of the ODE for $V_1$ in \eqref{fbp2} is
\begin{equation}
\label{V1}
V_1(\varphi) = C_1\varphi^{\beta_1-1}+ C_2\,\varphi^{\beta_2-1},
\end{equation}
where $C_1$ and $C_2$ are constants and $\beta_1\in(0,1)$ and $\beta_2<0$ are solutions of the quadratic equation
\[\frac{1}{2}\omega^2\beta(\beta-1)+\sigma\omega\beta+\mu_0=0.\]
The third and fourth boundary conditions in \eqref{fbp2} can be used to determine $C_1$ and $C_2$ as
\begin{equation*}
C_{1}=\frac{(1-\beta_{2})(1+\epsilon)}{\beta_{1}-\beta_{2}}B^{1-\beta_{1}} \quad \textrm{and} \quad C_{2}=
\frac{(\beta_{1}-1)(1+\epsilon)}{\beta_{1}-\beta_{2}}B^{1-\beta_{2}}.
\end{equation*}
From the condition $V_1(A)=1$ and the derived expressions for $C_{1}$ and $C_{2}$, we arrive at the equation
\begin{equation}\label{eqn:AB}
  (1-\beta_{2})\left(\frac{A}{B}\right)^{\beta_{1}-1}+(\beta_{1}-1)\left(\frac{A}{B}\right)^{\beta_{2}-1}=
\frac{\beta_{1}-\beta_{2}}{1+\epsilon}.
\end{equation}

\begin{lemma}\label{prop:unique}
There exists a unique value of $A/B\in(0,1)$ satisfying \eqref{eqn:AB}.
\end{lemma}

\begin{proof}
Letting $h(z)=(1-\beta_{2})z^{\beta_{1}-1}+(\beta_{1}-1)z^{\beta_{2}-1}-(\beta_{1}-\beta_{2})/(1+\epsilon)$ we can see that $h^{\prime}(z)=(1-\beta_{2})(\beta_{1}-1)\left[z^{\beta_{1}-2}-z^{\beta_{2}-2}\right]> 0$ for $z\in(0,1)$ since $\beta_{1}\in(0,1)$ and $\beta_{2}<0$. Furthermore, $\lim_{z\downarrow 0}h(z)=-\infty$ and $h(1)=\epsilon(\beta_{1}-\beta_{2})/(1+\epsilon)>0$.
Hence we conclude that there is a unique root of $h(z)=0$ in $(0,1)$.
\end{proof}

Next, the general solution of the ODE for $V$ in \eqref{fbp4} is
\begin{equation}
\label{V}
V(\varphi)=D_1 \varphi^{\beta_1} + D_2 \varphi^{\beta_2}
\end{equation}
for constants $D_1$ and $D_2$.
The second and third boundary conditions in \eqref{fbp4} can be used to determine $D_1$ and $D_2$ as
\begin{equation*}
D_{1}=\frac{A^{-\beta_{1}}}{\beta_{1}-\beta_{2}}\left[-\beta_{2}+(1-\beta_{2})A\right] \quad
\textrm{and} \quad D_{2}=\frac{A^{-\beta_{2}}}{\beta_{1}-\beta_{2}}\left[\beta_{1}+(\beta_{1}-1)A\right].
\end{equation*}
From the boundary condition $V'(B-)=1+\ep$ and the derived expressions for $D_1$ and $D_2$, we arrive at the equation
\begin{align*}
  (1+\epsilon)(\beta_{1}-\beta_{2})B &= \left(A/B\right)^{-\beta_{2}}\beta_2\left[\beta_{1}+(\beta_{1}-1)A\right]\\
  &\quad -\left(A/B\right)^{-\beta_{1}}\beta_1\left[\beta_{2}+(\beta_{2}-1)A\right]. \nonumber
\end{align*}
Denoting the unique root of \eqref{eqn:AB} as $\delta=A/B\in(0,1)$ we set $A=\delta B$ to obtain
\begin{align}
\label{eqn:bound}
  (1+\epsilon)(\beta_{1}-\beta_{2})B &= \delta^{-\beta_{2}}\beta_2\left[\beta_{1}+(\beta_{1}-1)\delta B\right]\\
 \notag
&\quad -\delta^{-\beta_{1}}\beta_1\left[\beta_{2}+(\beta_{2}-1)\delta B\right].
\end{align}
The linear equation \eqref{eqn:bound} has the unique solution
\begin{equation}
\label{eqn:bound2}
  B=\frac{\beta_1\beta_2(\delta^{-\beta_2}-\delta^{-\beta_1})}{(1+\ep)(\beta_1-\beta_2)-\beta_2(\beta_1-1)\delta^{1-\beta_2}+\beta_1(\beta_2-1)\delta^{1-\beta_1}},
\end{equation}
and it is straightforward to check that $B>0$, using that $\delta\in(0,1)$ in both the numerator and denominator.

From the above we see that $A$ and $B$ are uniquely determined by \eqref{eqn:AB} and \eqref{eqn:bound2}, and the corresponding candidate values $V_1$ and $V$ are
given by \eqref{V1} and \eqref{V}, respectively. Notice that in order to define $V$ on $(0,+\infty)$ we simply extend $V$ constructed above, in a $C^1$ way, by taking
\begin{align}
\label{eq:Vext1}&V(\varphi)=V(B)+(1+\ep)(\varphi-B),&\text{for $\varphi\ge B$,}\\
\label{eq:Vext2}&V(\varphi)=1+\varphi,&\text{for $\varphi\le A$}.
\end{align}
% Moreover we extend $V$ to $(0,A]$ in a $C^1$ way by taking $V(\varphi)=1+\varphi$. Finally,
Moreover, we extend $V_1$ to $[B,+\infty)$ in a $C^1$ way and to $(0,A]$ in a continuous way by taking
\begin{align}
\label{eq:V1ext1}&V_1(\varphi)=1+\ep,&\text{for $\varphi\ge B$},\\
\label{eq:V1ext2}&V_1(\varphi)=1,&\text{for $\varphi\le A$}.
\end{align}

\begin{theorem}\label{prop:final}
Let $A<B$ be the unique solution of \eqref{eqn:AB} and \eqref{eqn:bound2}, and let $V_1$ and $V$ be constructed as in \eqref{V1} and \eqref{V} with \eqref{eq:Vext1}--\eqref{eq:V1ext2}.
Denote $V_0(\varphi):=V(\varphi)-\varphi V_1(\varphi)$ and recall $\Phi^B$ and $\tau_A$ from \eqref{refl-Phi} and \eqref{tauA}.
Let $\Gamma^*=(0,\Gamma^B)$, let $\gamma^*_\theta$ be the
randomised stopping time generated by $\Gamma^*$, and set $\tau^*:=\tau_A$.
Then the randomised stopping pair $(\tau^*,\gamma^*_\theta)$ is a Nash equilibrium for the agent-form game with linear payoffs as in \eqref{linp} (i.e., a saddle point for the ex-ante game).
Moreover, for all $(x,\varphi)\in\R_+\times\R_+$ we have
\begin{align*}
\JJ_{x,\varphi}(\tau^*,\gamma^*_\theta)&=x V(\varphi),\\
\cJ^0_{x,\varphi}(\tau^*,0)&=x V_0(\varphi),\\
\cJ^1_{x,\varphi}(\tau^*,\Gamma^B)&=x V_1(\varphi).
\end{align*}
\end{theorem}

\begin{proof}
The proof relies on showing that $u(x,\varphi):=xV(\varphi)$ and $u^i(x,\varphi):=x V_i(\varphi)$, $i=0,1$,
fulfill all conditions in Theorem~\ref{prop:verifPDE}.

Let us start by setting, for $i=0,1$,
\[\C:=\{(x,\varphi)\,:\,u(x,\varphi)>(1+\varphi) x\}\quad\text{and}\quad\C^i:=\{(x,\varphi)\,:\,u^i(x,\varphi)<(1+\ep) x\}\]
and $\cS:=\R_+^2\setminus\C$, $\cS^i:=\R^2_+\setminus\C^i$.
From Lemma \ref{lem:V1} and the second equation in \eqref{fbp2} we obtain $\C^1=\R_+\times(0,B)$ and $\cS^1=\R_+\times[B,+\infty)$. Similarly, from \eqref{eq:Vp} and Lemma \ref{lem:V} we get $\C=\R_+\times(A,+\infty)$ and $\cS=\R_+\times(0,A]$.

Since $V$ and $V_1$ solve \eqref{fbp4}--\eqref{eq:Vp} and \eqref{fbp2}, respectively, it is immediate to check that $V_0$ solves \eqref{fbp1}--\eqref{V0-B}. Moreover, Lemma \ref{lem:V0} guarantees that $V_0$ also satisfies \eqref{xV0}. Then \eqref{V0le1} holds as well, implying $\C^0=\R_+^2$ and $\cS^0=\varnothing$.

Now that $\C$, $\C^i$, $\cS$, $\cS^i$ are specified, it is easy to check that on $\cS$ we have
\begin{align}\label{eq:vL}
\LL^0 u(x,\varphi)=\LL^0[x(1+\varphi)]=x(\mu_0+\mu_1\varphi)\le 0,
\end{align}
where the last inequality follows from Corollary~\ref{cor:A} (recall that $\LL^i$ is the infinitesimal generator of $(X,\Phi)$ under the measure $\P^i$). Therefore, \eqref{eq:vL} and \eqref{fbp4} imply \eqref{PDEu}. Moreover,
\eqref{fbp1} and \eqref{fbp2} imply \eqref{PDEui}. Furthermore, the second equations in \eqref{fbp1} and in \eqref{fbp2} imply \eqref{cont0}, and \eqref{V0-B} and the third equation of \eqref{fbp2} imply \eqref{smooth0}. Finally, $u^i(x,\varphi)\le x(1+\epsilon)$ for $i=0,1$ by \eqref{V0le1} and Lemma \ref{lem:V1}.

It is clear that $(X_t,\Phi^B_t)_{t\ge 0}$ meets conditions \eqref{SK2}--\eqref{SK1} by construction
since all probability measures we consider are equivalent on $\cF_t$, for each $t<\infty$.
Moreover, $\P^i(\tau_A<\infty)=1$ since $\tau_A$ is the first hitting time of a constant level for a reflected diffusion,
so it follows that $\P(\tau_A<\infty)=(1-\pi)\P^0(\tau_A<\infty)+\pi\P^1(\tau_A<\infty)=1$.

It only remains to check the transversality condition \eqref{tr2}. First we notice that $\tau_n(u)$ and $\tau_n(u^i)$, $i=0,1$, defined as in \eqref{tau-n} converge to infinity as $n\to\infty$ under $\P^i$ and $\PP^i$, $i=0,1$, thanks to the regularity of $u$ and $u^i$. Using that $\Phi^B_t\in[0,B]$ for all $t\ge 0$, $\P^0_{x,\varphi}$-a.s.~and that $V_0$ is bounded by one, we obtain
\begin{align*}
0\le&\,\lim_{n\to+\infty}\E^0_{x,\varphi}\Big[X_{\tau_n}V_0(\Phi^B_{\tau_n})\mathds{1}_{\{\tau_A>\tau_n\}}\Big]
\le \lim_{n\to+\infty}\E^0_{x,\varphi}\Big[X_{\tau_n}\mathds{1}_{\{\tau_A>\tau_n\}}\Big]=0
\end{align*}
since the $\P^0$-geometric Brownian motion $\{X_t, t\geq 0\}$ is uniformly integrable.
Thus \eqref{tr2} holds for $i=0$.

To prove \eqref{tr2} for $i=1$ we see that it follows from \eqref{fbp2} and an application of Ito's formula that
$Z_t:=e^{\mu_1(t\wedge\tau_A)}V_1(\Phi^B_{t\wedge\tau_A})$ is a
$\PP^1_{\varphi}$-martingale. By Fatou's lemma,
\[\EE^1_\varphi\Big[e^{\mu_1 \tau_A}\Big]\leq \lim_{t\to\infty}\EE^1_\varphi\Big[e^{\mu_1 (t\wedge\tau_A)}
V_1(\Phi^B_{t\wedge\tau_A})\Big]
\leq V_1(\varphi)<\infty,\]
which also implies $\PP^1(\tau_A<+\infty)=1$, as needed below.
Finally, we have
\begin{align*}
0\le&\,\lim_{n\to+\infty}\E^1_{x,\varphi}\Big[X_{\tau_n}V_1(\Phi^B_{\tau_n})\mathds{1}_{\{\tau_A>\tau_n\}}\Big]\\
\leq&\, (1+\ep) \lim_{n\to+\infty}\lim_{t\to+\infty} \E^1_{x,\varphi}\Big[X_{\tau_n}\mathds{1}_{\{\tau_A\wedge t>\tau_n\}}\Big]\\
=&\, x(1+\ep) \lim_{n\to+\infty}\lim_{t\to+\infty}\EE^1_{\varphi}\Big[e^{\mu_1 \tau_n} \mathds{1}_{\{\tau_A\wedge t>\tau_n\}}\Big]\\
\le &\,x(1+\ep)  \lim_{n\to+\infty}\EE^1_{\varphi}\big[ e^{\mu_1 \tau_n}\mathds{1}_{\{\tau_A>\tau_n\}}  \big]=0,
\end{align*}
where we used that $\Phi^B_t\in[0,B]$ for all $t\ge 0$, $\PP^1_{\varphi}$-a.s.~and
that $V_1$ is bounded by $1+\ep$ on $(0,B]$, and the last equality is due to dominated convergence.
\end{proof}

Recalling that the {\em ex-ante} value in our problem is $xV(\varphi)$ and $x\ge 0$, one final observation concerns its convexity with respect to $\varphi$ (since $V'\ge 0$, this is also equivalent to convexity with respect to $\pi$). The result is consistent with the existing literature on games with asymmetric information, going back to \cite{AM}, and it is in line with the most recent results in, e.g., \cite{C}, \cite{CR},  \cite{Gens}, \cite{GG} and \cite{G}.
\begin{proposition}\label{prop:convex}
The map $\varphi\mapsto V(\varphi)$ is convex on $[0,\infty)$.
\end{proposition}

\begin{proof}
The result could be derived by the explicit expression for $V$ but it would require checking the sign of all the constants involved. We follow an alternative approach that exploits the uniqueness of the couple $(A,B)$ solving \eqref{eqn:AB} and \eqref{eqn:bound2}.

First, we observe that if there exist $A\le \varphi_1<\varphi_2\le B$ such that $V''(\varphi_1)=V''(\varphi_2)=0$, then $V''=0$ on the interval $(\varphi_1,\varphi_2)$. This can be easily deduced by the maximum principle, upon noticing that $V''=:\hat v$ solves
\[
\tfrac{\omega^2\varphi^2}{2} \hat v''(\varphi)+(2\omega^2+\sigma\omega)\varphi \hat v'(\varphi)+(\omega^2+\sigma\omega+\mu_1)\hat v(\varphi)=0
\]
on $(\varphi_1,\varphi_{2})$ by direct differentiation of \eqref{bp}. Since $1=V'(A+)<V'(B-)=1+\epsilon$ and $V'>1$ on $(A,B)$ we conclude that $\varphi\mapsto V'(\varphi)$ may change monotonicity at most once. In particular, it is non-decreasing up to its global maximum, i.e., for all $\varphi\in(A,\varphi_0]$ with
\[\varphi_0:=\inf\{\varphi\in[A,B]: V'(\varphi_0)=\max_{\varphi\in[A,B]}V'(\varphi)\}.\]
If  $V'(\varphi_0)>1+\ep$, then
there exists $B_0\in(A,\varphi_0]$ such that $V'(B_0)=1+\epsilon$ and the couple $(A,B_0)$ is also a solution of \eqref{eqn:AB} and \eqref{eqn:bound2}, by the same construction used to arrive to those equations. This contradicts uniqueness of the solution.
Therefore, $V'(\varphi_0)=1+\ep$ and $\varphi_0=B$. Then $V'$ is monotone and $V$ is convex.
\end{proof}

\section{Numerical Results}\label{sec:numeric}

In this section we illustrate the value of the ex-ante game and the Nash equilibrium in the agent-form game found in Section~\ref{sec:linear}. We consider a base-case set of parameters with $\mu_0=-1$, $\mu_1=1$, $\sigma=0.5$ and $\epsilon=0.1$.
For these parameters, the boundaries defined by \eqref{eqn:AB} and \eqref{eqn:bound2} are found to be $A=0.329$ and $B=0.868$.
For ease of interpretation, however, in the following we will return to the posterior probability process $\Pi^*$, where we denote the lower boundary as $a:=A/(1+A)$ and the upper (reflecting) boundary as $b:=B/(1+B)$.
For our base case this corresponds to $a=0.248$ and $b=0.465$. Furthermore, we let $x=1$ and note that
$u(1,\varphi)=V(\varphi)$ and $u^i(1,\varphi)=V_i(\varphi)$; accordingly, we refer to $V$ and $V_i$ as value functions.  Finally, since $\mu=\mu_01_{\{\theta=0\}}+\mu_1 1_{\{\theta=1\}}$ we will stop referring to $\theta$ and use directly the events $\{\mu=\mu_0\}$ and $\{\mu=\mu_1\}$, which have a clearer intuitive meaning.

Firstly, Figure~\ref{fig:sim} demonstrates a typical sample path of the $\Pi^*$-process and its associated $\Gamma^{*,1}$-process.
Note that in this particular example, Player~1 stops at $\tau^*\approx 0.06$, and that $\Gamma^{*,1}_{\tau^*}\approx 0.13$. Consequently, Player~1 stops before Player~2 if either $\mu=\mu_0$ or if $\mu=\mu_1$ and
the uniformly distributed randomisation device $\mathcal U$ takes a value larger than 0.13.

\begin{figure}[htp!]\centering
  \psfrag{t}{$t$}
  \psfrag{pi}{$\Pi^*$}
  \psfrag{Gamma}{$\Gamma^{*,1}$}
  \subfigure[]{\includegraphics[width=0.45\textwidth]{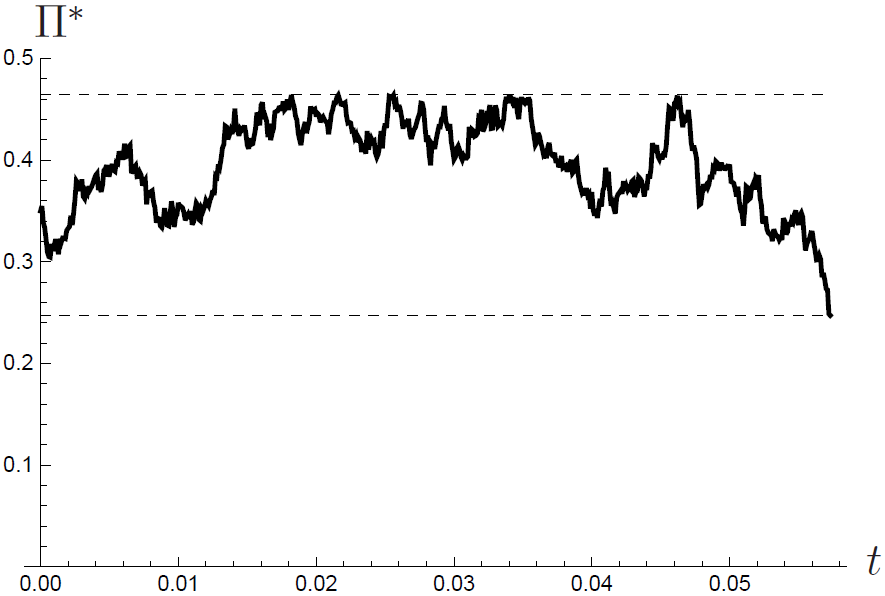}\label{fig:sim.pi}}
  \subfigure[]{\includegraphics[width=0.45\textwidth]{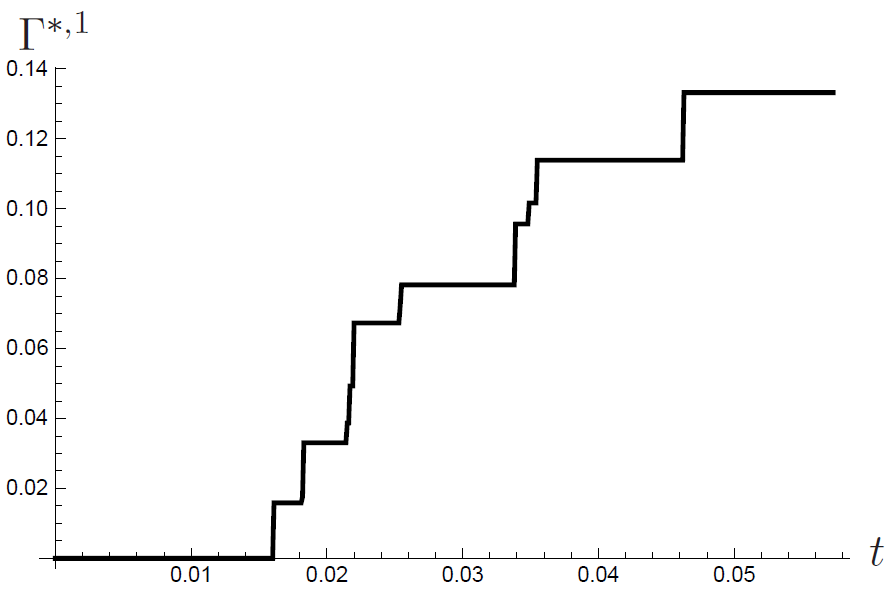}\label{fig:sim.gambpi}}
  \caption{A typical sample path of the $\Pi^*$-process (left) and its associated $\Gamma^{*,1}$-process (right) for our base-case parameters. Note that the dashed lines on the left represent the optimal boundaries $a=0.248$ and $b=0.465$ and that we have chosen $\pi=0.35$.} \label{fig:sim}
\end{figure}

Next, Figure~\ref{fig:value} shows the value functions for Player~1 and Player~2 corresponding to our base case.
Note that $V$ and $V_1$ satisfy smooth fit conditions at $a$ and $b$ respectively, and $V_0$ satisfies the reflection condition at $b$.
We also observe the properties of $V_0$ described in \eqref{V0-B} and \eqref{V0le1}, along with the properties of $V_1$ and $V$ described in Lemmas 5.5 and 5.6, respectively.
Convexity of the ex-ante value $V$, as described in Proposition \ref{prop:convex}, is also observed.
When the true drift is $\mu_1$, the informed player expects to pay out considerably more than Player~1 has reason to believe, and when the true drift is $\mu_0$, the informed player expects to pay out less.
When $\mu=\mu_1$, the gap between $1+\epsilon$ and the value of the game to Player~2 can be seen to represent the reduction in Player~2's expected cost due to Player~1 being uninformed.
Similarly, when $\mu=\mu_0$, the gap between 1 and the value of the game to Player~2 represents the reduction in Player~2's expected cost due to Player~1 being uninformed.

\begin{figure}[htp!]\centering
  \psfrag{V}{Value}
  \psfrag{pi}{$\pi$}
  \includegraphics[width=0.7\textwidth]{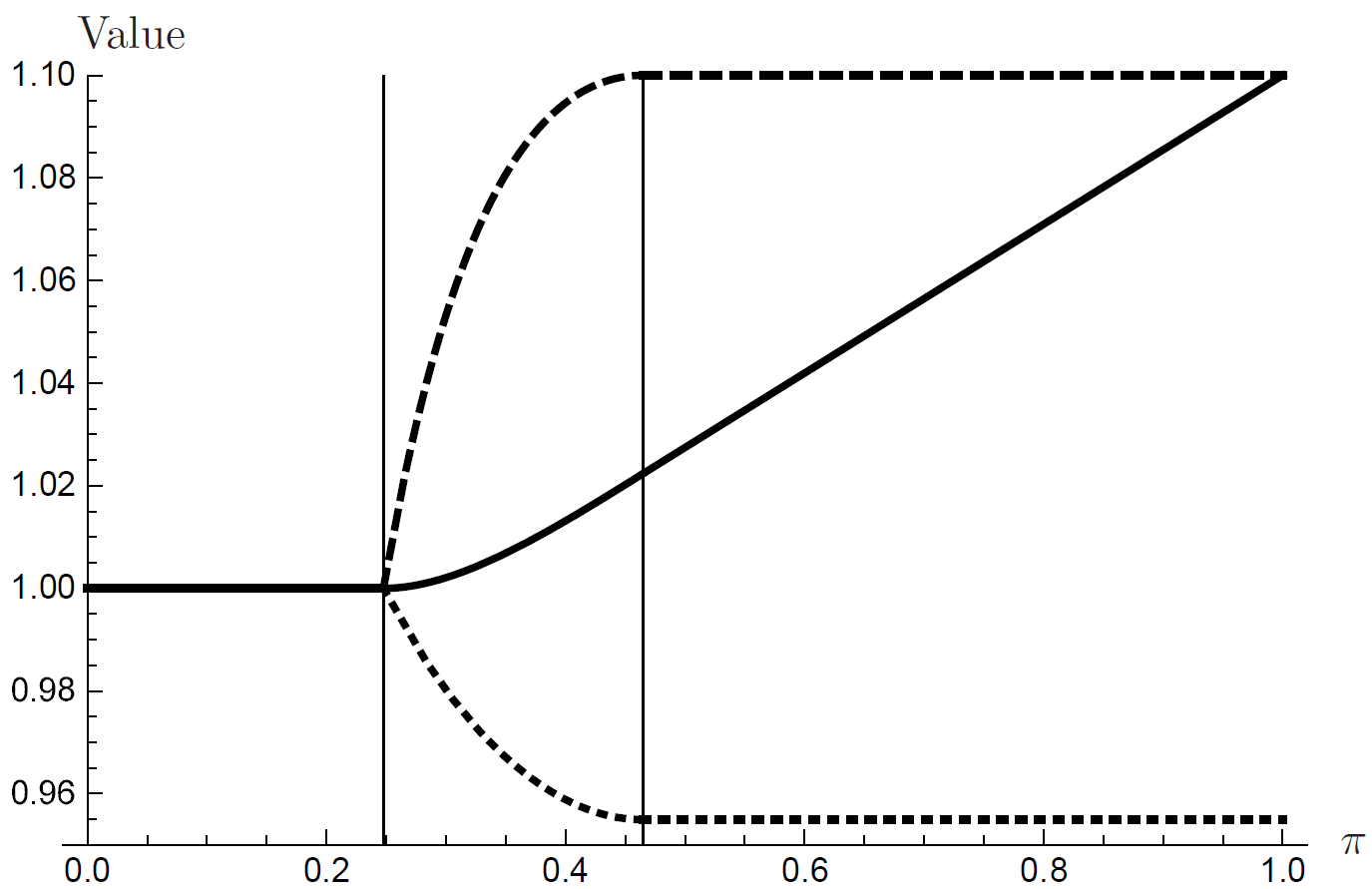}
  \caption{The value of the game to Player 1 (solid line; $(1-\pi)V_0+\pi V_1$) along with the value of the game to Player~2 when $\mu=\mu_0$ (dotted line; $V_0$) and $\mu=\mu_1$ (dashed line; $V_1$). The base-case parameters are $\mu_0=-1$, $\mu_1=1$, $\sigma=0.5$ and $\epsilon=0.1$; therefore $a=0.248$ and $b=0.465$ (represented by the two vertical lines).}\label{fig:value}
\end{figure}

Figure~\ref{fig:ab} shows comparative static results for the changing of all four parameters $(\mu_0,\mu_1,\sigma,\epsilon)$ with the base case used above.
We first note that the signal-to-noise ratio, $\omega=(\mu_1-\mu_0)/\sigma$, plays a crucial role in understanding these results since a higher $\omega$ will result in faster learning by the uninformed player.
In this sense, changes in the parameters $\mu_0$, $\mu_1$ and $\sigma$ will affect the signal-to-noise ratio and hence the speed of learning, which will ultimately have an impact on the equilibrium outcome.
Furthermore, changing $\mu_0$, $\mu_1$ and $\sigma$ will not only have an effect on the speed of learning (through the signal-to-noise ratio) but also
on the expected payoff of the game, potentially resulting in non-monotone dependencies due to these competing effects.
Finally, we note that $\epsilon$ only influences the problem through the payoff structure of the game and has no impact on the rate at which Player 1 is able to learn about the drift.
With this understanding in mind, we now proceed to describe the comparative statics results observed in Figure~\ref{fig:ab}.

\begin{figure}[htp!]\centering
  \psfrag{delta}{$\delta$}
  \psfrag{mu0}{$\mu_0$}
  \psfrag{mu1}{$\mu_1$}
  \psfrag{sigma}{$\sigma$}
  \psfrag{epsilon}{$\epsilon$}
  \psfrag{ab}{}
  \subfigure[]{\includegraphics[width=0.45\textwidth]{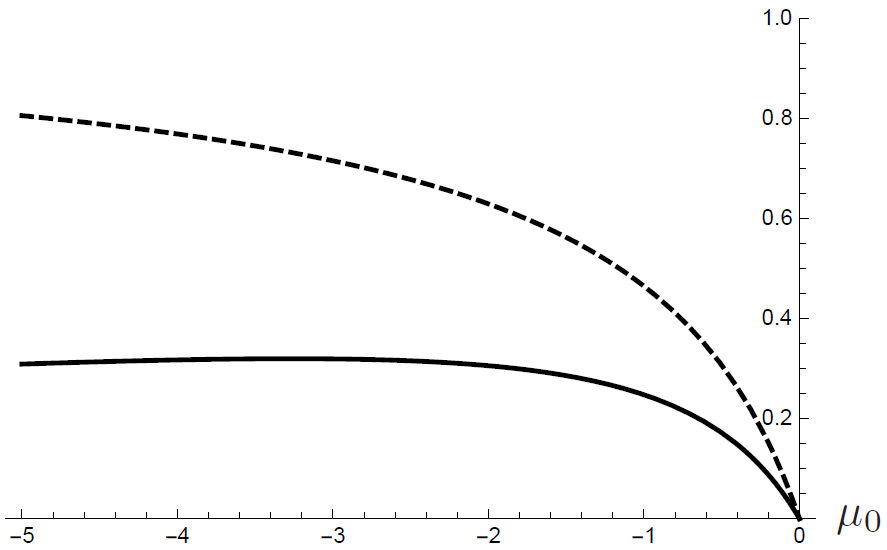}\label{fig:ab.mu0}}
  \subfigure[]{\includegraphics[width=0.45\textwidth]{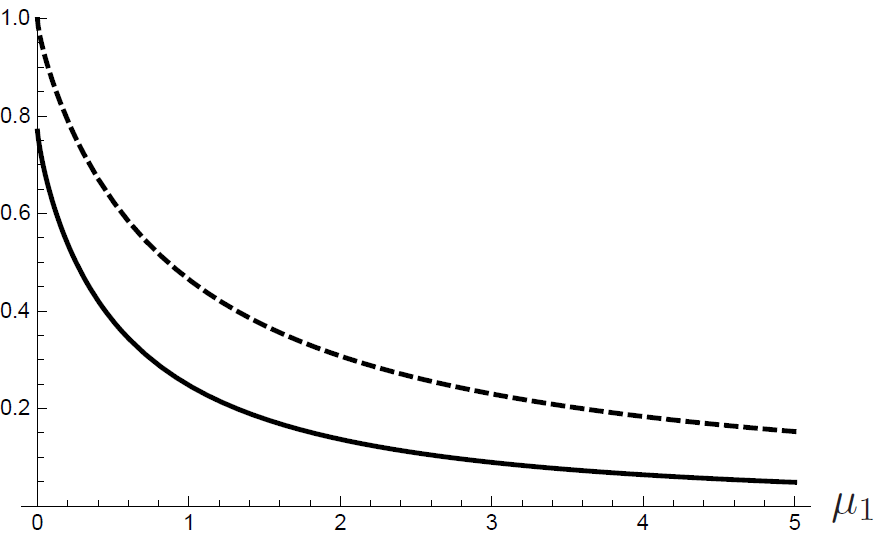}\label{fig:ab.mu1}}
  \subfigure[]{\includegraphics[width=0.45\textwidth]{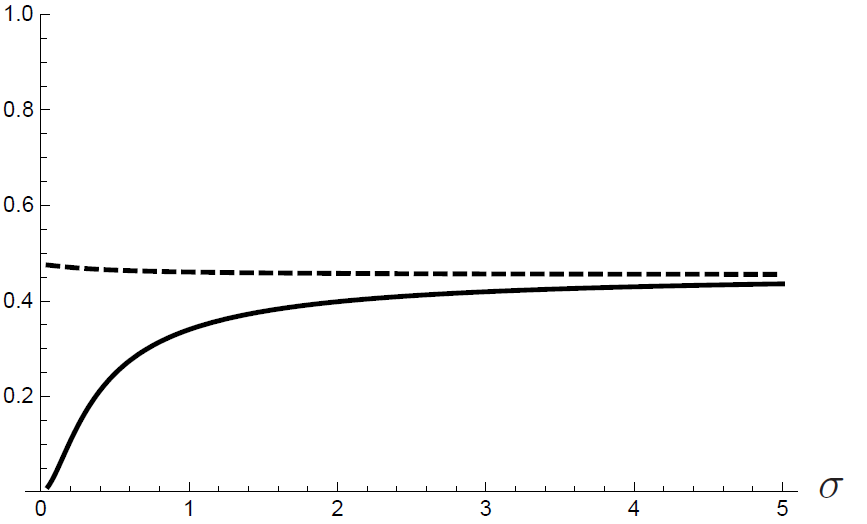}\label{fig:ab.sigma}}
  \subfigure[]{\includegraphics[width=0.45\textwidth]{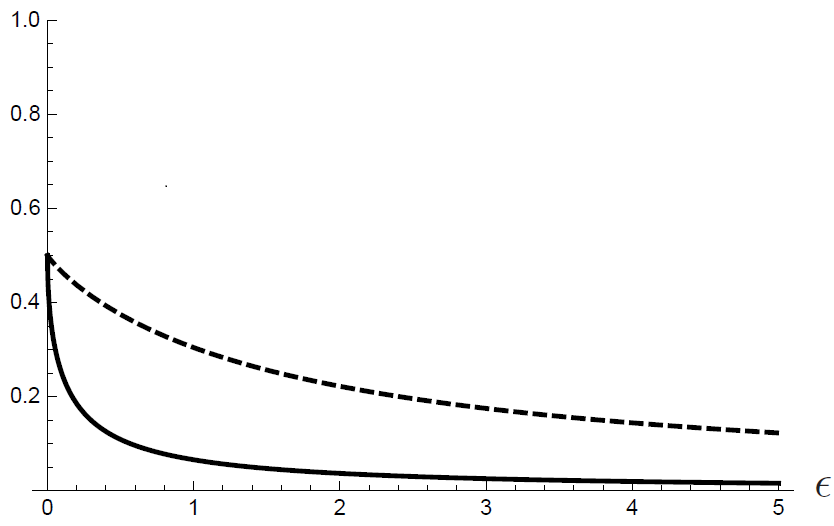}\label{fig:ab.epsilon}}
  \caption{The optimal boundaries ($a$ = solid line and $b$ = dashed line) for the base-case parameters ($\mu_0=-1$, $\mu_1=1$, $\sigma=0.5$ and $\epsilon=0.1$) as we vary $\mu_0$, $\mu_1$, $\sigma$ and $\ep$, respectively.} \label{fig:ab}
\end{figure}

We first consider the effect of changing $\mu_1$ on the equilibrium outcome. As $\mu_1$ increases (all else being equal), the good scenario for Player~1 gets better, both due to a larger drift, and also due to an increased signal-to-noise ratio which speeds up the learning process.
This indicates that the threshold $a$ should be decreasing in the drift $\mu_1$, which is also confirmed numerically, see Figure~\ref{fig:ab.mu1}.
Likewise, if $\mu=\mu_1$ and $\mu_1$ is large, then continuing is costly for Player~2, and at the same time, the advantage of having additional information about the drift is smaller (because of the increased signal-to-noise ratio). Consequently, the threshold $b$ should be decreasing in $\mu_1$, which is also confirmed numerically.

When considering a change in $\mu_0$, there are two competing effects on both players.
On one hand, a decreasing $\mu_0$ is bad for Player~1 (the sup-player), and hence has an increasing effect on the threshold $a$.
On the other hand, a decreasing $\mu_0$ increases the signal-to-noise ratio, which speeds up the learning process, and hence decreases $a$.
Figure~\ref{fig:ab.mu0} confirms the suspicion that there is no monotone dependence of $a$ on $\mu_0$.
For the same reasons as above, the effect of a change in $\mu_0$ on the upper threshold $b$ is ambiguous.
However, this potential ambiguity is not visible in Figure~\ref{fig:ab.mu0} for our base-case parameters.

From Figure~\ref{fig:ab.sigma} we see that as $\sigma$ increases the optimal threshold is increasing for Player~1 and
decreasing for Player~2.
The intuition behind this is that, as $\sigma$ increases, the signal-to-noise ratio decreases, resulting in slower learning and
hence a smaller value function for Player~1 and hence an increased $a$.
For Player 2, however, while an increased $\sigma$ means that they are better able to hide their information from Player~1 (an incentive to increase $b$), the reduced variance of the $\Pi$-process also means that first hitting time of a given threshold is larger for an increased $\sigma$.
Since $\mu_1>0$, a longer expected time to stop would ultimately result in an increased expected cost for Player 2 (an incentive to decrease $b$).
By the numerics, the net result for our base-case parameters is that Player~2 reduces their threshold $b$ as $\sigma$ increases.

Lastly, we consider the effect of a change in $\epsilon$.
Since the value of $\epsilon$ does not impact the ability of Player~1 to learn about the drift, its effect on the equilibrium can only be through the payoff structure of the game.
Therefore, all value functions clearly increase in $\ep$; for Player~1 this means that the continuation region is increasing in $\ep$, so that the threshold $a$ is decreasing.
However, no easy monotonicity for $b$ can be deduced as there is no obvious effect on the continuation region for Player~2 (since also the obstacle depends on $\ep$).
From Figure~\ref{fig:ab.epsilon} we observe the anticipated monotonic dependence of $a$ on $\ep$ and, for our base-case parameters at least, $b$ is also seen to be monotonic decreasing in $\ep$.

Finally, to calculate the {\em value of information} for the game,
we end the article with an informal discussion on the case with symmetric and incomplete information. Assume that both players have the same initial prior distribution for $\mu$, that is they agree on $\pi$ as the initial probability that the drift is $\mu_1$, and $1-\pi$ as the probability that the drift is $\mu_0$. Then randomisation is not needed for either player and a saddle point in stopping times $(\tau_1,\tau_2)$ can be obtained. In fact, the game with linear payoffs reduces to
\begin{eqnarray*}
  U(x,\varphi)&=& \frac{x}{1+\varphi} \sup_{\tau_1}\inf_{\tau_2}\EE^0_\varphi\left[e^{\mu_0\tau_1}(1+\Phi_{\tau_1})1_{\{\tau_1<\tau_2\}}\right.\\
&& \hspace{30mm} \left.+(1+\ep)e^{\mu_0\tau_2}(1+\Phi_{\tau_2})1_{\{\tau_2\leq\tau_1\}}\right]
\end{eqnarray*}
where
\[d\Phi_t=\sigma\omega\Phi_t\,dt + \omega\Phi_t\,d\W_t^0\]
under $\PP^0$. It is then straightforward to check that one can find $A,B\in(0,\infty)$ with $A<B$ and a function $\widehat{V}$ with $1+\varphi\leq\widehat{V}\leq(1+\ep)(1+\varphi)$ such that
\begin{eqnarray*}
\left\{\begin{array}{ll}
\frac{\omega^2\varphi^2}{2}\widehat{V}''(\varphi) + \sigma\omega\varphi\widehat{V}'(\varphi) + \mu_0\widehat{V}(\varphi)=0, &  \textrm{for } \varphi\in (A,B)\\
\widehat{V}(\varphi)=1+A, &  \textrm{for } \varphi\in(0,A]\\
\widehat{V}'(A+)=1, &\\
\widehat{V}(\varphi)=(1+\ep)(1+B), &  \textrm{for } \varphi\in[B,\infty)\\
\widehat{V}'(B-)=1+\ep. &
\end{array}\right.
\end{eqnarray*}
Using standard verification arguments, $(\tau^*_1,\tau_2^*):=(\tau_A,\tau_B)$ is a saddle point of stopping times, and the corresponding value function is given by $U(x,\varphi)=x\widehat{V}(\varphi)/(1+\varphi)$.

\begin{figure}[htp!]\centering
  \psfrag{V}{$U/x$}
  \psfrag{Pi}{$\pi$}
  \psfrag{VoI}{VoI}
  \includegraphics[width=0.45\textwidth]{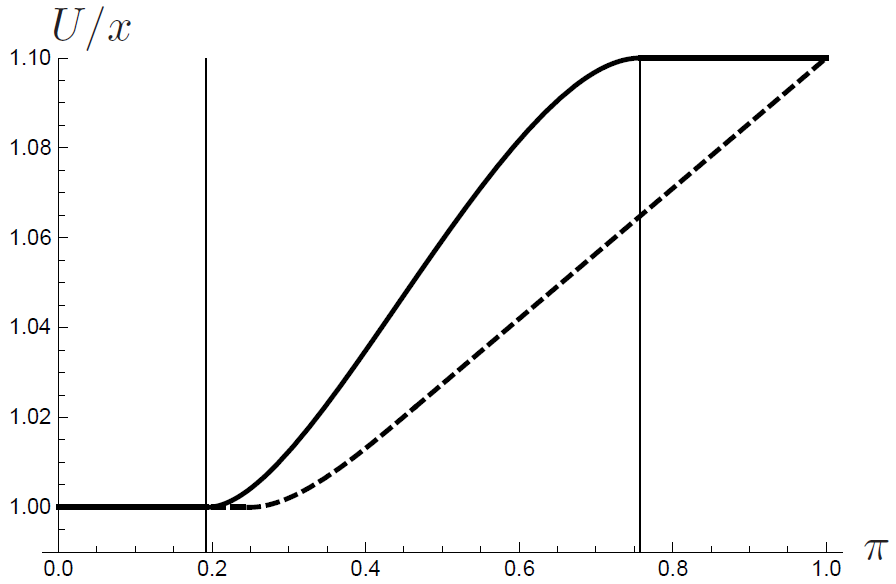}\hspace{2mm}
  \includegraphics[width=0.45\textwidth]{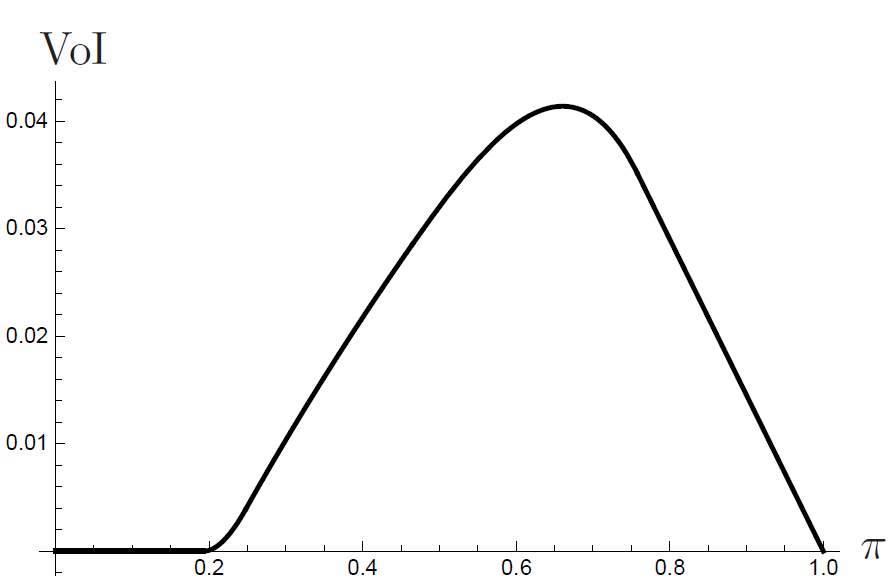}
  \caption{On the left: The common value function for both players in the symmetric incomplete information case (solid line) in comparison to the value function in the asymmetric case (dashed line). The two vertical lines correspond to the values $a:=A/(1+A)$ and $b:=B/(1+B)$ (for the symmetric case). On the right: The difference between these values, which represents the value of information in our game. Base-case parameters: $\mu_0=-1$, $\mu_1=1$, $\sigma=0.5$ and $\epsilon=0.1$, which yields $a=0.193$ and $b=0.758$ (for the symmetric case).} \label{fig:value-sym}
\end{figure}

Figure~\ref{fig:value-sym} plots the value function $U(1,\pi)=U(x,\pi)/x$ for our base-case parameters, along with the value function of the uninformed player for the asymmetric case for comparison.
The difference between the asymmetric value function and the symmetric one is also plotted and can be interpreted as the value of information in this setting.

\end{document}